\documentclass[11pt]{amsart}
\usepackage{latexsym}
\usepackage{amsfonts,amssymb,amsmath,amsthm,amscd}
\usepackage[mathscr]{eucal}
\usepackage{graphicx}
\usepackage{subfigure}

\newcommand{\g}{\mbox{${\mathfrak g}$}}

\newcommand{\kf}{\mbox{${\mathfrak k}$}}

\newcommand{\p}{\mbox{${\mathfrak p}$}}

\newcommand{\C}{\mbox{${\mathbb C}$}}
\newcommand{\HH}{\mbox{${\mathbb H}$}}
\newcommand{\I}{\mbox{${\mathbb I}$}}

\newcommand{\PP}{\mbox{${\mathbb P}$}}

\newcommand{\R}{\mbox{${\mathbb R}$}}

\newcommand{\tr}{{\rm tr}}
\newcommand{\ric}{{\rm Ric}}

\makeatletter
\def\numberwithin#1#2{\@ifundefined{c@#1}{\@nocnterrr}{%
  \@ifundefined{c@#2}{\@nocnterr}{%
  \@addtoreset{#1}{#2}%
  \toks@\expandafter\expandafter\expandafter{\csname the#1\endcsname}%
  \expandafter\xdef\csname the#1\endcsname
    {\expandafter\noexpand\csname the#2\endcsname
     .\the\toks@}}}}
\makeatother
\numberwithin{equation}{section}

\newtheorem{thm}[equation]{Theorem}
\newtheorem{lemma}[equation]{Lemma}
\newtheorem{prop}[equation]{Proposition}

\newtheorem{cor}[equation]{Corollary}
\newtheorem{ex}[equation]{Example}

\newtheorem{rem}[equation]{Remark}
\newenvironment{rmk}{\begin{rem} \em}{\end{rem}}

\begin{document}

\title{A Family of Steady Ricci Solitons and Ricci-flat Metrics}
\author{M. Buzano}
\address{Department of Mathematics and Statistics, McMaster
University, Hamilton, Ontario, L8S 4K1, Canada}
\email{maria.buzano@gmail.com}
\author{A. S. Dancer}
\address{Jesus College, Oxford University, OX1 3DW, United Kingdom}
\email{dancer@maths.ox.ac.uk}
\author{M. Gallaugher}
\address{Department of Mathematics and Statistics, McMaster
University, Hamilton, Ontario, L8S 4K1, Canada}
\email{gallaump@mcmaster.ca}
\author{M. Wang}
\address{Department of Mathematics and Statistics, McMaster
University, Hamilton, Ontario, L8S 4K1, Canada}
\email{wang@mcmaster.ca}
\thanks{M. Wang is partially supported by NSERC Grant No. OPG0009421}

\date{revised \today}

\begin{abstract}
We produce new non-K\"ahler complete steady gradient Ricci solitons
whose asymptotics combine those of the Bryant solitons and the
Hamilton cigar. We also obtain a family of complete Ricci-flat metrics
with asymptotically locally conical asymptotics. Finally, we obtain
numerical evidence for complete steady soliton structures
on the vector bundles whose distance sphere bundles are respectively
the twistor and ${\rm Sp}(1)$ bundles over quaternionic projective space.
\end{abstract}

\maketitle


\bigskip
\setcounter{section}{-1}

\section{\bf Introduction}

In this article we continue the study of Ricci solitons using
 methods of dynamical systems, focusing on the
case of steady solitons. A Ricci soliton consists of
a complete Riemannian metric $g$ and a (complete) vector field $X$
satisfying the equation:
\begin{equation} \label{LieRS}
{\ric}(g) + \frac{1}{2} \,{\sf L}_{X} g + \frac{\epsilon}{2} \, g = 0
\end{equation}
where $\epsilon$ is a real constant and $\sf L$ denotes the Lie derivative.
The soliton is {\em steady} if $\epsilon =0$, {\em expanding}
if $\epsilon >0$ and {\em shrinking} if $\epsilon < 0$.
If $X$ is Killing, then $g$ is Einstein and the soliton is
called {\em trivial}. By the work of Perelman \cite{Per}, nontrivial steady and
expanding solitons must be noncompact.

A Ricci soliton is of {\em gradient type} if the vector field $X$ is the
gradient of a globally defined smooth function, referred to as the
{\em  soliton potential}. Many examples of K\"ahler gradient Ricci solitons
of all three types exist in the literature, see, for example, \cite{Ko},
\cite{Cao}, \cite{FIK}, \cite{WZh}, \cite{ACGT}, \cite{PS}, and  \cite{DW1}.
By contrast, fewer non-K\"ahler gradient Ricci solitons are known.

In \cite{DW2} a family of non-K\"ahler steady gradient Ricci solitons was
constructed generalising the rotationally symmetric Bryant solitons \cite{Bry} on
$\R^n \, (n > 2)$ and Ivey's related examples \cite{Iv}. The manifolds in this family
consisted of warped products on an arbitrary number of Einstein factors with positive
Einstein constants, and they exhibited asymptotically paraboloid geometry, like the
examples of Bryant and Ivey. In particular, they gave examples of steady
gradient Ricci solitons in dimensions greater than three which are not rotationally
symmetric. For dimension $3$, Brendle has recently proved that the Bryant soliton is the
only non-trivial $\kappa$-noncollapsed steady gradient Ricci soliton \cite{Br1}.

In this paper we use methods similar to those in \cite{DW2}, but allow one
of the factors in the warped product to be a circle (hence with zero Einstein
constant). We obtain complete steady Ricci solitons whose asymptotics are a mixture
of the paraboloid Bryant asymptotics and the cylindrical asymptotics
of Hamilton's cigar soliton \cite{Ha1}. More precisely, the metric
on the circle factor is asymptotically constant while those
on the other factors asymptotically grow like the geodesic distance
coordinate. This type of asymptotics has been observed previously for
K\"ahler steady solitons, cf \cite{DW1}. The general results of
Buzano \cite{Buz} now allow us to dispense with some of the analysis of
\cite{DW2} concerning smoothness at the other end of the manifold,
where the circle collapses. It should be mentioned that the special case
in which there are two factors (including the circle factor) was discussed
in Ivey's Duke thesis and in \cite{Iv}.

For dimensions greater than three, Brendle has obtained an analogue of his
$3$-dimensional rigidity theorem for complete steady gradient Ricci solitons \cite{Br2}.
Under the hypotheses of positivity of sectional curvatures and of being
``asymptotically cylindrical", he proves that such a soliton must be the
Bryant soliton. We note that our examples always have some negative
sectional curvatures when the number of positive Einstein factors
is at least $2$, although the Ricci tensor is non-negative. As for the
asymptotically cylindrical property, our examples do satisfy the upper
and lower scalar curvature bounds, but not the stronger requirement
involving the Gromov-Hausdorff convergence of rescaled flows to shrinking
cylinders.

In addition we prove some general results about steady solitons of cohomogeneity
one type, including monotonicity and concavity results for the soliton
potential, and decay estimates for the ambient scalar curvature and the
mean curvature of the hypersurfaces.

We also study a family of solutions to our equations that yield
complete Ricci-flat metrics. These are related to examples of
B\"ohm \cite{Bo1} which are multiple warped products whose factors are
Einstein metrics of positive scalar curvature. In our examples, as in the
soliton case above, one of these factors is replaced by a circle.
The resulting equations are rather different in character from those
considered by B\"ohm, due to the fact that we no longer have a solution
representing a cone over a positive scalar curvature Einstein metric on the
hypersurface (which acts as an attractor for the B\"ohm system). Note
that the special case in which there are only two factors is explicitly
integrable, and was discussed in \cite{BB} and \cite{Be}. This special case
includes the Riemannian Schwarzschild metric.

Combining our construction with the work of Boyer, Galicki and Koll{\'a}r
on Einstein metrics on exotic spheres,  the work of K. Kawakubo and R. Schultz,
and the recent work of Hill, Hopkins and Ravenel, we deduce that in
all dimensions congruent to $3$ mod $4$ other than $3, 7, 15, 31, 63$ and
possibly $127$, there are homeomorphic but not diffeomorphic complete
non-compact Ricci-flat manifolds as well as steady gradient Ricci solitons
(cf Corollary \ref{diffstr}).

The above examples all fall within the class of multiple warped products
on Einstein factors with nonnegative Einstein constant.
The analysis of the dynamical system in such cases is aided by the
fact that the scalar curvature of the hypersurface is bounded below.
In the examples treated in \cite{DW2} and (in the Einstein case) \cite{Bo1},
the scalar curvature is in fact strictly positive. This is related to the fact
that the Lyapunov function defined in \cite{DW2} for a general
cohomogeneity one steady soliton system is in these cases actually a positive
definite quadratic form (up to an additive constant). This gives coercive
estimates on the flow which facilitate the analysis. In the examples of the
present paper, where one factor in the warped product is flat, the
Lyapunov is no longer definite, but it becomes definite upon restriction
to a subsystem of one dimension lower, and this turns out to be enough
for many of the arguments to go through.

It is therefore of great interest to consider the soliton equation in cases
where the hypersurfaces have more complicated scalar curvature. One natural
class of examples are hypersurfaces which are the total spaces of Riemannian
submersions for which the hypersurface metric involves two functions, one
scaling the base and one the fibre of the submersion. If the fibre is a circle
this leads us to ans\"atze familiar in the Einstein case from the work of
Calabi and B\'erard Bergery. In the soliton case, K\"ahler examples are known
(see the references above), although the general case is still not well-understood.
For higher-dimensional fibres the equations are more complicated. In the Einstein case
B\"ohm obtained existence results for compact examples in some low dimensions \cite{Bo2}.

In the last section we discuss some results from a numerical study of the steady soliton
equation for these more complicated principal orbit types. Viewing the quaternionic
projective space as a quaternionic K\"ahler manifold, we take its twistor bundle
(resp. canonical ${\rm Sp}(1)$ bundle) as the principal orbits in the associated
$\R^3$ (resp. $\R^4$) bundle over $\HH\PP^m, m \geq 1$. We produce numerical evidence
of complete steady gradient Ricci soliton structures on these bundles. Note that
the existence of complete Ricci-flat metrics on these bundles, including ones
with special holonomy, was considered by \cite{BryS}, \cite{GPP} and \cite{Bo1}.
In the soliton case, however, we did not detect any difference between low and
high dimensional cases.

\section{\bf Generalities on cohomogeneity one steady solitons}

In \cite{DW1} two of the authors set up the formalism for Ricci solitons of cohomogeneity one.
More precisely, we considered the situation of a manifold $M$ with an open dense set
foliated by equidistant diffeomorphic hypersurfaces $P_t$ of real dimension $n$.
In other words, the metric is taken to be of the form $\bar{g}=dt^2 + g_t$ where $g_t$
is a metric on $P_t$ and $t$ is the arclength coordinate along a geodesic orthogonal
to the hypersurfaces. This formalism is somewhat more general than the cohomogeneity one
ansatz, as it allows us to consider metrics with little or no symmetry provided that
appropriate additional conditions on $P_t$ are satisfied, see the following as well as
Remarks 2.18 and 3.18 in \cite{DW1}.

We shall consider solitons of {\em gradient type}, that is, we take $X = {\rm grad} \; u$
for a function $u$. Equation (\ref{LieRS}) then becomes
\begin{equation} \label{gradRS}
{\ric}(\bar{g}) + {\rm Hess}(u) + \frac{\epsilon}{2} \, \bar{g} = 0.
\end{equation}
We will further suppose, in line with the cohomogeneity one formalism, that
$u$ is a function of $t$ only, and treat it as both a smooth function on
the manifold and a function of the single variable $t$.

We let $r_t$ denote the Ricci tensor of $g_t$, viewed as an endomorphism via $g_t$.
Then we can define $L_t$, the shape operator of the hypersurfaces, by the equation
$\dot{g_t} = 2 g_t L_t$. We assume that the scalar curvature $S_t=\tr(r_t)$ and the
mean curvature $\tr(L_t)$ (with respect to the normal $\nu=\frac{\partial}{\partial t}$)
are constant on each hypersurface. We shall often in the future suppress the $t$-dependence
in the above tensors.

In this setting, the above equation becomes the system (cf \S 1 of \cite{DW1})
\begin{eqnarray}
  -{\rm tr} (\dot{L}) - {\rm tr} (L^2) + \ddot{u} + \frac{\epsilon}{2} &=& 0,  \label{TT} \\
    r - ({\rm tr}\, L)L - \dot{L} + \dot{u} L + \frac{\epsilon}{2} \, \I &=& 0, \label{SS}\\
d ({\rm tr} L) + \delta^{\nabla} L &=& 0.  \label{TS}
\end{eqnarray}
The first two equations represent the components of the equation
in the $\frac{\partial}{\partial t}$ direction and in the directions
tangent to $P$, respectively. Also, $\delta^{\nabla} L$ denotes the
codifferential for $TP$-valued $1$-forms, and the third equation represents
the equation in mixed directions.

The above assumptions are satisfied, for example, if $M$ is
of cohomogeneity one with respect to an isometric Lie group action.
They are satisfied also when $M$ is a multiple warped product over an
interval, which is the situation we focus on in this paper.

In the warped product case the final equation involving the
codifferential automatically holds. This is also true for
cohomogeneity one metrics that are {\em monotypic}, i.e. when there
are no repeated irreducible summands in the isotropy representation
of the principal orbits (cf \cite{BB}, Prop. 3.18).

We have a conservation law
\begin{equation} \label{cons1}
\ddot{u} + (-\dot{u} + {\rm tr}\, L) \dot{u}  -\epsilon u = C
\end{equation}
for some constant $C$.
Using the equations this may be rewritten as
\begin{equation}
{\rm tr} (r_t) + {\rm tr} (L^2) - (\dot{u} - {\rm tr}\, L )^2 - \epsilon u +
\frac{1}{2}(n-1) \epsilon = C.
\end{equation}

The term ${\rm tr} (r_t)$ is the scalar curvature $S$ of the principal
orbits. Recall that if $\bar{R}$ denotes the scalar curvature of the
ambient metric, from now on written as $\bar{g} = dt^2 + g_t$, then
$$ \bar{R} = -2 {\rm tr} (\dot{L}) - {\rm tr} (L^2) - ({\rm tr} L)^2 +S. $$
We can deduce the equality
\begin{equation}\label{ham}
\bar{R} + \dot{u}^2 + \epsilon u = -C -\frac{\epsilon}{2}(n+1),
\end{equation}
which is just the cohomogeneity one version of Hamilton's
identity $\bar{R} + |\nabla u|^2 + \epsilon u =\;$ constant.

\smallskip

We now specialise to the case of {\em steady solitons}, that is, $\epsilon =0$.
The conservation law is now
\begin{equation} \label{cons2}
{\rm tr} (r_t) + {\rm tr} (L^2) - (\dot{u} - {\rm tr}\,  L)^2 = C.
\end{equation}

In \cite{DW2} the following result was proved.
\begin{prop}
The function $(\dot{u} - {\rm tr}\, L)^{-2}$
is a Lyapunov function, that is, it is monotonic on each interval on which it is defined.
\end{prop}

\medskip
The expression $-\dot{u} + {\rm tr} \, L$ is the soliton version of the
hypersurface mean curvature. It occurs so frequently in our analysis
that we shall introduce some special notation for it
\[
\xi := -\dot{u} + {\rm tr} \, L.
\]

\begin{rmk}
The conservation law (\ref{cons2}) shows that our Lyapunov function is a constant multiple
of
\[
\frac{{\rm tr}(r_t) + {\rm tr} (L^2)}{(\dot{u} - {\rm tr}\; L)^2}  -1.
\]
\end{rmk}
It is convenient to
define a function ${\mathcal L } = \frac{C}{\xi^2}$ and we refer to this
as {\em the} Lyapunov.

It is often useful to define a new independent variable $s$ by
\begin{equation} \label{st}
\frac{d}{ds} := \frac{1}{\xi} \frac{d}{dt} =
\sqrt{\frac{\mathcal L}{C}} \frac{d}{dt}
\end{equation}
and use a prime to denote $\frac{d}{ds}$. We shall see presently that $\xi$ is always
positive in the case of a complete steady soliton.
 Another useful quantity is
\[
{\mathcal H} = \frac{{\rm tr} L}{\xi} = 1 + \frac{\dot{u}}{\xi}
= 1 +u^\prime,
\]
which was introduced in \cite{DW3}, \cite{DHW}.

In the steady case the locus $\{ {\mathcal L} =0, {\mathcal H}=1 \}$
is invariant under the flow, and trajectories in this region of
phase space correspond to trivial solitons, i.e., ones where
$u$ is constant and $g$ is an Einstein metric. Analogous
statements hold in the expanding and shrinking cases if we modify
$\mathcal L$ appropriately. We refer the reader to \cite{DW3}
and \cite{DHW} for further discussion.

\bigskip

We now describe some general results about complete cohomogeneity one steady
solitons of gradient type. Some of these results can be deduced from theorems
about general solitons found in e.g., \cite{FR}, \cite{MS}, \cite{Wu}, \cite{WeiWu}.
However, in the cohomogeneity one situation, the statements sometimes take on a
stronger or more precise form, which will be useful for checking asymptotic
behaviour in numerical studies. We have also included their proofs here. Besides
being more elementary, they involve ideas which are useful for analysing
existence questions in the cohomogeneity one case.

We shall be looking at complete noncompact steady solitons with one special
orbit, in which case  we may assume, without loss of generality, that $t \in [0, \infty)$
and the special orbit occurs at $t=0$. We let $k$ denote the dimension of the
collapsing sphere at $t=0$.

The equation (\ref{ham}) becomes
\begin{equation} \label{ham2}
\bar{R} + (\dot{u})^2 = -C
\end{equation}
for complete steady solitons. We recall the result of B. L. Chen \cite{Chen}
that $\bar{R} \geq 0$ for complete steady solitons, with equality iff $\bar{g}$
is Ricci-flat. Hence we deduce the important inequality
\[
C < 0
\]
for non-trivial steady solitons.
Note that this is a global consequence of completeness which does not follow
from examining local existence in some neighbourhood of a singular orbit
(cf. \cite{Buz}). The specific value of $C$ is unimportant as it can be changed
by a positive multiple via a homothety of the soliton metric.

\begin{prop} \label{steady-pot}
The soliton potential is strictly decreasing and strictly concave
on $(0, \infty)$.
\end{prop}

\begin{proof}
Let $t_0$ be a critical point of $u$ in $(0, \infty)$. The
conservation law (\ref{cons1}) with $\epsilon =0$, together with the
negativity of $C$, show that $u$ is strictly concave in a neighbourhood of $t_0$.
So the critical points of $u$ are isolated and nondegenerate.

Next let $t_0 < t_1$ be two consecutive critical points. The concavity
statement means there must be a critical point between $t_0$ and $t_1$, a
contradiction. So if a critical point exists it is unique.

The smoothness conditions at the special orbit imply
$\dot{u}(0)=0$ and $\xi = \frac{k}{t} + O(t)$ near $t=0$. Substituting into
(\ref{cons1}) yields $(k+1) \ddot{u}(0) = C<0$ so in fact we
have concavity at $t=0$ also.
Hence the above argument shows there are in fact no critical points
of $u$ in $(0, \infty)$ and $u$ is hence strictly decreasing.

Now set $y = \dot{u}$ and differentiate (\ref{cons1}); using (\ref{TT})
we obtain
\[
\ddot{y} + \xi \dot{y} - {\rm tr}(L^2) y =0.
\]
We know $y \leq 0$ on $[0, \infty)$ with equality only at $0$. Also
$\dot{y}(0)<0$ and $y^2 < -C$. If $t_0 >0$ is a critical point of $y$, then
$\ddot{y}(t_0) \leq 0$ with equality only if $L(t_0)=0$. But this
implies, by (\ref{cons1}), that $-\dot{u}(t_0)^2 =C$, which contradicts
positivity of $\bar{R}$. Hence $y$ is strictly concave at its critical points.
Looking at the first critical point and using the above information on $y$
now shows no critical points can exist. So $\ddot{u} = \dot{y}$ is negative
for all  $t>0$.
\end{proof}

\begin{prop} \label{steady-MC}
The mean curvature ${\rm tr} \; L$ is strictly decreasing
and satisfies $0 < {\rm tr} \; L \leq \frac{n}{t}$. The generalised
mean curvature $\xi$ is strictly decreasing and tends to $\sqrt{-C}$
as $t$ tends to $\infty$.
\end{prop}

\begin{proof}
The preceding proposition shows that $\frac{d}{dt}({\rm tr} \; L) =
-{\rm tr} (L^2) + \ddot{u}$ is negative. By Cauchy-Schwartz, we have
$\frac{d}{dt} ({\rm tr} \; L) < -\frac{1}{n}( {\rm tr} \; L)^2$.

Suppose ${\rm tr} \; L$ vanishes at $t_0$. Then ${\rm tr} \; L < 0$ on
$(t_0, \infty)$.
Integrating the inequality $\frac{d}{dt}({\rm tr} \; L) \leq \ddot{u}$
from $t_0 + \delta$ to $t$, where $\delta>0$, yields
\[
\frac{\dot{v}}{v} = ({\rm tr}\; L)(t) \leq ({\rm tr}\; L)(t_0 + \delta)
+ \dot{u}(t) - \dot{u}(t_0 + \delta) < \dot{u}(t) - \dot{u}(t_0 + \delta)
\]
where $v(t)$ is the volume of the metric $g_t$ relative to a fixed
invariant background metric on the principal orbit.
Since $\dot{u}$ is strictly decreasing, and bounded by (\ref{ham2}),
it tends to a negative constant $-a$ as $t$ tends to $\infty$. For sufficiently
large $t$ we may assume that
\[
\dot{u}(t) + a < \frac{1}{2}( \dot{u}(t_0 + \delta) +a ) := \alpha
\]
So $\frac{\dot{v}}{v}$ is less than the negative constant $-\alpha$, which
implies the metric has finite volume, a contradiction to Theorem 1.11 in
\cite{MS}.

It follows that $\tr \; L$ never vanishes and hence is positive
everywhere, since it tends to $+ \infty$ at $t=0$. Now for $0 < \eta < T$
we have
\[
\int_{t= \eta}^{t= T}  \frac{d({\rm tr}\; L)}{({\rm tr}\; L)^2}
< -\frac{1}{n} \int_{t =\eta}^{t = T} dt.
\]
Therefore
\[
\frac{1}{({\rm tr}\; L)(\eta)} - \frac{1}{({\rm tr}\; L)(T)}
< -\frac{1}{n} (T - \eta)
\]
which, on letting $\eta \rightarrow 0,$ gives us our claimed upper bound
on ${\rm tr} \; L$.

As $\dot{u}$ is bounded below and decreasing,
$\lim_{t \rightarrow \infty} \dot{u}\, {\rm tr} \, L =0$. Now the conservation law
(\ref{cons1}) implies that $\lim_{t \rightarrow \infty} \ddot{u}$ exists. Since
$\ddot{u} < 0$,  the boundedness of $\dot{u}$ implies that
 $\lim_{t \rightarrow \infty} \ddot{u} = 0$.
The conservation law then yields $a = \sqrt{-C}$.

Finally, by Eq. (\ref{TT}), $\dot{\xi} = -\tr (L^2) =-\tr ((L^{(0)})^2) -\frac{1}{n} (\tr \,L)^2 < 0$
since we have shown that $\tr \;L > 0$. The limiting value of $\xi$ is then that of $-\dot{u},$
i.e., $\sqrt{-C},$ as $\tr \;L$ tends to $0$.
\end{proof}

\begin{rmk}
We have shown for complete steady gradient Ricci solitons of cohomogeneity one
with a special orbit at one end that $\dot{u}$ tends to a negative
constant and $\ddot{u}$ tends to $0$ as $t$ tends to $\infty$. So the soliton
potential $u$ will have asymptotically linear behaviour. In numerical searches
it is of course not possible to generate solutions over an infinite interval,
so the asymptotic behaviour of quantities such as the soliton potential provides
a valuable check that a soliton has in fact been found numerically. We shall consider
numerical examples in \S 5.
\end{rmk}

\begin{cor}
The ambient scalar curvature $\bar{R}$ is decreasing and tends to zero
as $t$ tends to $\infty$. Furthermore, we have
$$ 0 < -\dot{u}\,\tr \,L < \bar{R} < 2\sqrt{-C}\left(\frac{n}{t}\right) + \frac{n^2}{t^2}. $$
\end{cor}
\begin{proof}
By (\ref{ham2}), $\frac{d}{dt} \bar{R} = -2 \dot{u} \ddot{u} <0$.
The limiting value follows from the above proposition and (\ref{ham2}).

Next, using (\ref{cons1}) followed by (\ref{ham2}), we obtain
$\frac{d}{dt} \bar{R} = 2 \dot{u} (\bar{R} +  \dot{u}\, \tr \, L). $
Since $\frac{d}{dt} \bar{R} < 0$ and $\dot{u} < 0$, we deduce the lower bound.
For the upper bound, note that by the conservation law (\ref{cons2}) and
Proposition \ref{steady-MC}, we have $S + \tr(L^2) = \xi^2 +C > 0$.  On the
other hand, from the trace of Eq. (\ref{SS}) and  (\ref{cons2}) we have
$$S + \tr(L^2) = -\bar{R} + (\tr \, L)^2 - 2\dot{u}\, \tr\, L.$$
Therefore, using Proposition \ref{steady-MC} again, we deduce that
$$ \bar{R} < (\tr \,L)(\tr \,L - 2\dot{u}) \leq \frac{n}{t} \left(\frac{n}{t} + 2 \sqrt{-C}\right).$$
\end{proof}

\begin{rmk}
Note that from the limiting values of $\dot{u}$ and $\ddot{u}$ we also get
$\lim_{t \rightarrow \infty} \frac{d}{dt} \bar{R} = 0.$
The asymptotic behaviour of $\bar{R}$ for general steady solitons is given by
Theorem 3.4 in \cite{FR}, from which the asymptotic value of $\bar{R}$ also follows.
The upper bound for $\bar{R}$ above is somewhat stronger than what can be deduced from
the general upper bound given in Corollary 1.3 of \cite{Wu}. The upper bound
shows that for $t \geq 1$, $\bar{R} < (2\sqrt{-C}+n)\,\frac{n}{t}$, which is
independent of the principal orbit. It is unclear whether/when $\tr \; L$ has an
asymptotic lower bound of the form $\frac{\rm const}{t}$. This is an interesting
question, however, in view of the hypotheses in Brendle's rigidity result \cite{Br2}.
\end{rmk}



Finally, we discuss another Lyapunov function, a modification of which will play
an important role in \S 4. This function, denoted by ${\mathscr F}_0$ below, was first
considered  by C. B\"ohm in \cite{Bo1} for the Einstein case and was subsequently
studied in \cite{DHW} for the soliton case.

\begin{cor} \label{steady-F}
Let ${\mathscr F}_0$ denote the function
$v^{\frac{2}{n}}\left( S + \tr( (L^{(0)})^2 ) \right)$ defined on the velocity phase space
of the cohomogeneity one gradient Ricci soliton equations, where $L^{(0)}$ is the trace-free part of $L$.
Then ${\mathscr F}_0$ is non-increasing along the trajectory of a complete non-trivial steady soliton.
Furthermore, the function ${\mathcal F} :=v^{\frac{2}{n}}\left( S + \tr( L^2 ) \right)$ is strictly
decreasing along such a trajectory.
\end{cor}
\begin{proof} By Proposition 2.17 in \cite{DHW}, we have the formula
\begin{equation} \label{Fdot}
 \dot{\mathscr F}_0 = -2 v^{\frac{2}{n}} \ \tr((L^{(0)})^2) \left(\xi - \frac{1}{n} \tr L \right).
\end{equation}
However, by Propositions \ref{steady-pot} and \ref{steady-MC},
$$ \xi - \frac{1}{n} \tr \,L = -\dot{u} + \left(1 - \frac{1}{n} \right) \tr \, L > 0.$$
So ${\mathscr F}_0$ is strictly decreasing except where $L$ is a multiple of the identity.

As for the second statement, since $\tr( L^2) = \tr( (L^{(0)})^2 ) + \frac{1}{n} (\tr L)^2$,
it suffices to examine
\begin{eqnarray*}
      \frac{d}{dt} \left( v^{\frac{2}{n}} (\tr L)^2 \right) & = & \left(\frac{2}{n}\right) v^{\frac{2}{n}} (\tr L)^3
               + 2 v^{\frac{2}{n}} (\tr L)(\tr L)^{\cdot} \\
           & = & 2 v^{\frac{2}{n}} (\tr L) \left( \frac{1}{n}(\tr L)^2 - \tr(L^2) + \ddot{u} \right) \\
           & \leq & 2 v^{\frac{2}{n}} (\tr L)\, \ddot{u},
\end{eqnarray*}
where we have used (\ref{TT}) and the Cauchy-Schwartz inequality. By Propositions \ref{steady-pot}
and \ref{steady-MC} the last quantity is negative along the trajectory.
\end{proof}

\section{\bf Multiple warped products}
We now specialise to multiple warped products, that is
metrics of the form
\begin{equation} \label{metric}
dt^2 + \sum_{i=1}^{r} g_i^2(t)\,  h_i
\end{equation}
on $I \times M_1 \times \cdots \times M_r$\, where $I$ is an interval in
$\mathbb R,$ $r \geq 2,$ and $(M_i, h_i)$ are Einstein manifolds
with real dimensions $d_i$ and Einstein constants $\lambda_i$. Note that
$n=\sum_i d_i \geq 3$ once some $M_i$ is non-flat.

The shape operator and Ricci endomorphism are now given by
\begin{eqnarray*}
L &=& {\rm diag} \left( \frac{\dot{g_1}}{g_1} \,\I_{d_1}, \cdots,\frac{\dot{g_r}}{g_r} \,\I_{d_r}\right)  \\
r &=& {\rm diag} \left( \frac{\lambda_1}{g_1^2} \, \I_{d_1}, \cdots,\frac{\lambda_r}{g_r^2}\, \I_{d_r}\right)
\end{eqnarray*}
where $\I_m$ denotes the identity matrix of size $m$.
As in \cite{DHW}, we work with the variables
\begin{eqnarray}
X_i &=& \frac{\sqrt{d_i}}{\xi} \frac{\dot{g_i}}{g_i} \label{def-Xi}\\
Y_i &=& \frac{\sqrt{d_i}}{\xi} \frac{1}{ g_i} \label{def-Yi}
\end{eqnarray}
for $i=1, \ldots, r$.
Notice that the definition of $Y_i$ in \cite{DW2} and \cite{DW3} differs from ours
by a scale factor of $\sqrt{\lambda_i}$; the new choice is more appropriate
to our current situation where one of the $\lambda_i$ may be zero.
Now
\[
\sum_{j=1}^{r} \, X_j^2  = \frac{{\rm tr} (L^2)}{\xi^2} \quad\textup{and}\quad
\sum_{j=1}^{r} \, \lambda_j Y_j^2 = \frac{ {\rm tr} (r_t)}{\xi^2}.
\]
So the Lyapunov function becomes
\begin{equation} \label{Lyap}
{\mathcal L} := \frac{C}{\xi^2} = \sum_{i=1}^{r} \, (X_i^2 + \lambda_i Y_i^2)  - 1
\end{equation}
where $C$ is a nonzero constant.

As mentioned above, we introduce the new coordinate $s$ defined by
(\ref{st}) and use
 a prime ${ }^{\prime}$ to denote differentiation with respect to $s$.

In our new variables the Ricci soliton system (\ref{TT})-(\ref{SS}) with $\epsilon=0$
becomes
\begin{eqnarray}
X_{i}^{\prime} &=& X_i \left( \sum_{j=1}^{r} X_j^2 -1 \right) +
\frac{\lambda_i Y_i^2}{\sqrt{d_i}} \, , \label{eqnX} \\
Y_{i}^{\prime} &=&  Y_i \left( \sum_{j=1}^{r} X_j^2 -\frac{X_i}{\sqrt{d_i}}
\right) \label{eqnY}
\end{eqnarray}
for $i=1, \ldots, r$. Note that homothetic solutions of the system (\ref{TT})-(\ref{TS})
give rise to the same solution of the above system.

We shall be concerned exclusively with the multiple warped situation
for the rest of the paper. Recall that in this case equation (\ref{TS})
is automatically satisfied. Note also that the above equations imply the equation
\begin{equation} \label{eqnLyap}
{\mathcal L}^{\prime} = 2 {\mathcal L} \left( \sum_{i=1}^{r} X_i^2 \right),
\end{equation}
so ${\mathcal L} =0$ is flow-invariant.
We also use the notation ${\mathcal G} := \sum_{i=1}^{r} X_i^2$ as this
quantity often occurs in our calculations.

The quantity ${\mathcal H} = \frac{{\rm tr} L}{\xi}$ becomes
$\sum_{i=1}^{r} \sqrt{d_i} X_i$ in our new variables.
We have the equation
\[
({\mathcal H} - 1)^\prime = ({\mathcal H} -1 )({\mathcal G}-1) + \mathcal L
\]
so, as mentioned above, we see that the region
$\{ {\mathcal L} =0, {\mathcal H} =1 \}$ of phase space
corresponding to Ricci-flat metrics is flow-invariant.

While in \cite{DW2} all $\lambda_i$ were taken to be positive, that is, the
Einstein constants on each $M_i$ were positive, we shall now look at the
case where the collapsing factor $M_1$ is $S^1$, so $d_1=1$, $\lambda_1=0,$
and the remaining $\lambda_i$ are positive. Then the equation for $X_1$ becomes:
\[
X_{1}^{\prime} = X_1 \left( \sum_{j=1}^{r} X_j^2 -1 \right).
\]
Note in particular this means the locus $X_1 =0$ is flow-invariant.

Conversely, if we have a solution of the above system
(\ref{eqnX}), (\ref{eqnY}) with $\lambda_1=0$, in the
region ${\mathcal L} <0$ (so $C<0$), we may recover $t$ and the
metric components $g_i$ from
\begin{equation} \label{def-tgi}
dt = \sqrt{\frac{\mathcal L}{C}}\,\, ds, \ \ \ \ \
g_i = \frac{\sqrt{d_i}}{Y_i} \, \sqrt{\frac{\mathcal L}{C}}\,\,.
\end{equation}
We can choose $t=0$ to correspond to $s = -\infty$.

The soliton potential is recovered from integrating
\begin{equation} \label{def-u}
 \dot{u} = \tr(L) - \sqrt{\frac{C}{\mathcal L}} =
\sqrt{\frac{C}{\mathcal L}}\, ({\mathcal H} - 1),
\end{equation}
and $\tr(L)$ is calculated using
\begin{equation} \label{logdiff-gi}
 \frac{\dot{g_i}}{g_i} = \sqrt{\frac{C}{\mathcal L}} \frac{X_i}{\sqrt{d_i}}\, .
\end{equation}

The following lemma is a routine calculation.
\begin{lemma} \label{zeros}
Let $d_1 =1$ and $d_i > 1$ for $i > 1$, so that
$\lambda_i =0$ iff $i=1$. The stationary points of $($\ref{eqnX}$)$, $($\ref{eqnY}$)$
are now:
\begin{enumerate}
\item[$($i$)$] the origin

\item[$($ii$)$] points with $Y_i=0$ for all $i$, and $\sum_{i=1}^{r} X_i^2 =1$

\item[$($iii$)$] points given by
\[
X_i = \sqrt{d_i}\, \rho_A \;\; : \;\; Y_i^2 = \frac{d_i}{\lambda_i}\,\, \rho_A (1- \rho_A),
\;\;\;  i \in A
\]
and $X_i = Y_i =0$ for $i \notin A$, where $A$ is any nonempty subset
of $\{2, \ldots, r \}$, and $\rho_A = \left( \sum_{j \in A} d_j \right)^{-1}$

\item[$($iv$)$] the line where $X_i=0$ for all $i$ and $Y_i=0$ for $i >1$

\item[$($v$)$] the line where $X_1 =1$ and $X_i, Y_i =0$ for $i > 1$.

\end{enumerate}
\end{lemma}

Note that $\mathcal L$ equals $-1$ in case (i) and (iv), and equals $0$ in case
(ii), (iii) and (v).   Cases (iv) and (v) are special to the case $d_1=1$
and mean that in this situation the origin is a non-isolated critical
point.

\section{\bf Soliton solutions}

As in \cite{DW2}, we shall construct complete non-compact steady
soliton metrics where one factor $M_1$ collapses at one end,
corresponding to $t=0$.  For the collapse to be smooth we take
$M_1$ to be a sphere $S^{d_1}$. (Note that $d_1$ is the same as the dimension $k$
in \S 1.) The manifold underlying the Ricci soliton is then the total space of
a trivial vector bundle of rank $d_1 + 1$ over
$M_2 \times \cdots \times M_r$.  In our case, of course, $d_1=1$.
The initial conditions for the soliton solution
to be $C^2$ are the existence of the following limits:
\begin{equation} \label{bdy0}
g_1(0)=0   \;\; : \;\; g_i(0) = l_i \neq 0 \; (i > 1),
\end{equation}

\begin{equation} \label{bdy1}
\dot{g_1}(0)=1 \;\; : \;\; \dot{g_i}(0) = 0 \; (i > 1),
\end{equation}

\begin{equation} \label{bdy2}
\ddot{g}_{1}(0) =0 \;\; : \;\; \ddot{g_i}(0) \;\, {\rm  finite} \;
(i > 1),
\end{equation}

\begin{equation} \label{bdyu}
u(0) \, \,  {\rm finite}: \, \, \dot{u}(0)=0 \, \, : \,\ddot{u}(0) \, \, {\rm \, finite}.
\end{equation}

In our $X_i, Y_i$ variables, this means we consider trajectories in
the unstable manifold of the critical point $P_0$ of (\ref{eqnX})
and (\ref{eqnY}) given by
\[
X_1 = 1, \;\; Y_1 = 1, \;\; X_i =Y_i=0  \,\, (i > 1).
\]
This critical point lies on the level set ${\mathcal L}=0$ of the Lyapunov.

The linearisation about this critical point is the system
\begin{eqnarray*}
x_1^{\prime} &=& 2x_1 \\
y_1^\prime &=& x_1 \\
x_i^\prime &=& 0 \;\;\; (i \geq 2) \\
y_i^\prime &=& y_i \;\; (i \geq 2)
\end{eqnarray*}
with eigenvalues
$2$,  $1$ ($r-1$ times), and $0$ ($r$ times).

\medskip
In contrast to the situation of \cite{DW2} we now have a centre manifold.

\medskip
The results of \cite{Buz} now show we have an $(r-1)$-parameter family
of trajectories $\gamma(s)$ such that $\lim_{s \rightarrow -\infty} \gamma(s)
= P_0$ and pointing into the region ${\mathcal L} < 0$. As in
\cite{DW2}, (\ref{eqnLyap}) shows that such trajectories stay in
${\mathcal L} < 0$. We can moreover choose the trajectories to have $Y_i >0$
for all time (note that the locus $Y_i =0$ is always invariant under the flow).

Because $d_1=1$ and hence $\lambda_1 =0$, the Lyapunov
${\mathcal L} = \sum_{i=1}^{r} X_i^2  + \sum_{i=2}^{r} \lambda_i Y_i^2$
does not involve $Y_1$, so the region ${\mathcal L} \leq 0$ is no
longer compact, in contrast to the situation in \cite{DW2}.

However, since $\lambda_1=0$, the variable $Y_1$ only enters
into the equations through the equation for $Y_1^\prime$. Hence
by omitting (\ref{eqnY}) for $i=1$ we obtain a subsystem
of (\ref{eqnX}), (\ref{eqnY}) for $X_i \,(i=1, \ldots, r)$ and
$Y_i \,(i=2, \ldots, r)$ and on this new space, ${\mathcal L} \leq 0$
is compact.  Moreover, once we have a solution to the subsystem
we can recover $Y_1$ via
\[
Y_{1}(s) = Y_{1}(s_0) \exp \left(\int_{s_0}^{s} \sum_{j=1}^{r} X_j^2 -X_1
\right).
\]
The critical points for the subsystem are given by cases (i), (ii)
and (iii) of the critical points for the full system.
In particular $P_0$ corresponds to the critical point $\hat{P}_0 =
(1,0,\ldots, 0)$ in the subsystem, and we have an $r-1$ parameter
family of solutions emanating from this point and lying in the
region $Y_i > 0$ and ${\mathcal L} < 0$.

\medskip
Let us now analyse these trajectories in ${\mathcal L} < 0$.
For the subsystem, where this region is precompact, all
the variables are bounded by $1$ and the flow exists for all $s$.
Hence this is true for the original flow also.

The arguments of Prop 3.7 of \cite{DW2} show that the flow in the subsystem
converges to the origin, so ${\mathcal L}$ converges to $-1$. We deduce
from (\ref{st}) that as $s$ tends to $\infty$, so does $t$, hence the metric
is complete.

The proof of Lemma 4.4 (i) in \cite{DW2} carries over to show that
all $X_i$ are positive on the trajectory.
Using the arguments of Lemma 3.8 in \cite{DW2} we can show, using the equation
\[
\left(\frac{X_i}{Y_i^2} \right)^\prime = \left(\frac{X_i}{Y_i^2} \right)
\left(-1 - {\mathcal G} + \frac{2X_i}{\sqrt{d_i}} \right) + \frac{\lambda_i}{\sqrt{d_i}}\;,
\]
the following result.
\begin{lemma}
We have $\lim_{s \rightarrow \infty} \frac{X_i}{Y_i^2} = \frac{\lambda_i}{\sqrt{d_i}}$ for $i \geq 2$. \qed
\end{lemma}
(Recall that the $Y_i$ in the current paper differ from those in
\cite{DW2} by a $\sqrt{\lambda_i}$ scale factor.)

\medskip

Now
\[
\frac{1}{2} \frac{d}{dt} (g_i^2) = g_i \dot{g_i} = \frac{d_i}{Y_i^2}
\frac{\mathcal L}{C} \frac{X_i}{\sqrt{d_i}}
\sqrt{\frac{C}{\mathcal L} } =
\frac{\sqrt{d_i} X_i}{Y_i^2} \sqrt{ \frac{\mathcal L}{C}}
\rightarrow \frac{\lambda_i}{\sqrt{|C|}}
\]
as $s$ tends to $\infty$. Hence we deduce that as $t$ tends to $\infty$,
$g_i^2$ to leading order asymptotically behaves like
$ \frac{2 \lambda_i t}{\sqrt{-C}} $ for $i > 1$.

For $i=1$, we again use the equation
\[
\left(\frac{X_1}{Y_1^2} \right)^\prime = \left(\frac{X_1}{Y_1^2} \right)
\left(-1 - {\mathcal G} + 2X_1 \right),
\]
which we may write in the form
\[
(\log \psi)^{\prime} = - 1 + \phi
\]
where $\psi = \frac{X_1}{Y_1^2}$ and $\phi$ tends to $0$ as $s$ tends
to $\infty$. Choosing $0< \delta <1$ and $s_0$ such that $|\phi(s)| < \delta$ for
$s > s_0$, we integrate and obtain
\[
-(1 + \delta)(s - s_0) < \log \frac{\psi(s)}{\psi(s_0)} <
(-1 + \delta)(s-s_0)
\]
for $s > s_0$. Exponentiating gives
\[
e^{-(1 + \delta)(s-s_0)} < \frac{\psi(s)}{\psi(s_0)}
< e^{-(1 - \delta)(s-s_0)}
\]
so $\psi = \frac{X_1}{Y_1^2}$
decays to zero as $s \rightarrow \infty$.
Now, as in the $i>1$ case, we have
\[
\frac{1}{2} \frac{d}{dt} (g_1^2) \; dt=
\frac{X_1}{Y_1^2} \sqrt{ \frac{\mathcal L}{C}} \; dt =
\frac{X_1}{Y_1^2}  \frac{\mathcal L}{C} \; ds,
\]
and this integrand is positive and dominated by
$\frac{1}{|C|} \frac{X_1}{Y_1^2}.$ Integrating
and using the exponential bound above shows that the
increasing function $g_1^2$ is bounded above, hence converges
to a positive limit ${\alpha}^2$.

\begin{rmk} \label{Y1value}
Since $g_1(t)$ tends to $\alpha$ and $\xi$ tends to $\sqrt{-C}$ as $t$ tends
to infinity, it follows from (\ref{def-Yi}) that $Y_1$ tends to a positive
constant as $s$ tends to infinity. This means that the soliton trajectory
in the full $X_i, Y_i$ space tends to one of the stationary points
of type (iv) (lying in ${\mathcal L} < 0$) in Lemma \ref{zeros}.

\end{rmk}

We have therefore deduced the following theorem.
\begin{thm} \label{asymptotics}
The metric corresponding to our trajectory has the form, to leading order in $t$
as $t \rightarrow +\infty$,
\[
dt^2 + {\alpha}^2 d \theta^2 + t \; {\sf h}_{\infty}
\]
where $\alpha$ is a positive constant, $\theta$ is the angle coordinate
on $M_1 = S^1$ and ${\sf h}_{\infty}$ is the product
Einstein metric on $M_2 \times \cdots \times M_r$.
The volume growth is asymptotically $t^{\frac{n+1}{2}}$. \qed
\end{thm}

\begin{rmk}
We thus obtain asymptotic behaviour which is a mixture of the
asymptotically paraboloid geometry of the Bryant solitons on $\mathbb R^n$
(for $n>2$) and the Hamilton-Witten cigar geometry on $\mathbb R^2$.
\end{rmk}

\begin{thm} \label{mainthm}
Let $M_2, \ldots, M_r$ be compact Einstein manifolds with positive scalar
curvature. There is an $r-1$ parameter family of non-homothetic complete smooth
steady Ricci solitons on the trivial  rank $2$ vector bundle
over $M_2 \times \ldots \times M_r$, with asymptotics given by
Theorem \ref{asymptotics}.  \qed
\end{thm}

\begin{rmk} \label{ricci}
As with the metrics of \cite{DW2}, we can show that our soliton metrics
have nonnegative Ricci curvature.
The sectional curvatures decay like $\frac{1}{t}$
or faster, and the curvatures $K(U_i \wedge U_j)$ where $U_i, U_j$
are tangent to $M_i, M_j$ respectively with $i,j \geq 2$ and $i \neq j$
are negative. The scalar curvature decays like $\frac{1}{t},$
and  satisfies $\frac{c_1}{t} \leq \bar{R} \leq \frac{c_2}{t}$ for
certain positive constants $c_1, c_2$ and all sufficiently large $t$.
(These constants depend only on $n$ and $\sqrt{-C}$.)
In particular, the asymptotic scalar curvature ratio
$\limsup_{d \rightarrow +\infty} {\bar{R}} d^2$,  where $\bar{R}$ is the
scalar curvature and $d$ is the distance from a fixed origin
in the manifold, is $+\infty$, as it should be.
\end{rmk}

\section{\bf Complete Ricci-flat metrics}

As mentioned in the introduction, a special case of solutions to the soliton
equations is that of trivial solitons, where the metric is Einstein
and the potential is constant. In the steady case, this means the
metric is Ricci-flat.

In \cite{Bo1} B\"ohm constructed an $r-2$ parameter family of complete
Ricci-flat metrics using warped products over $r$ Einstein manifolds
with positive Einstein constants. He assumed in his construction that the collapsing
Einstein factor is a sphere of dimension at least $2$. In this section we
will remove this dimension restriction, i.e., we produce
analogues of these metrics in the case where $M_1 = S^1$ (so $d_1 =1$).
The special case of $r=2$ was treated in \cite{BB} (see also p. 271 of \cite{Be}),
where an explicit solution was found. It includes the Riemannian Schwarzschild
solution, which is the special case when $M_2 = S^2$.

We recall from \S 2 that for trajectories representing Ricci-flat metrics, we
have ${\mathcal L}=0$ and ${\mathcal H}=1$. Therefore we need to study
trajectories emanating from the critical point $P_0$ and lying in the locus
${\mathcal L} = 0$ rather than going into the region ${\mathcal L} < 0$.
These form an $r-2$ parameter family.

We note that as ${\mathcal L} =0$ we have to modify our procedure to recover the
metric from solutions to (\ref{eqnX}) and (\ref{eqnY}). We now define $t$ by
\begin{equation} \label{stE}
dt = \exp \left( \int_{s*}^{s} \sum_{j=1}^{r} X_j^2 \right) \; ds
\end{equation}
for some fixed $s^*$. Also, let
\[
g_i = \frac{\sqrt{d_i}}{Y_i} \,\exp \left( \int_{s*}^{s} \sum_{j=1}^{r} X_j^2 \right)
\]
so
\[
\frac{\dot{g}_i}{g_i}\, \exp \left( \int_{s*}^{s} \sum_{j=1}^{r} X_j^2 \right)
= -\frac{Y_i^\prime}{Y_i} + \sum_{j=1}^{r} X_j^2
=\frac{X_i}{\sqrt{d_i}}
\]
and hence
\[
{\rm tr} L = \sum_{i=1}^{r} \; \frac{d_i \dot{g_i}}{g_i}
= {\mathcal H} \exp \left( -\int_{s*}^{s} \sum_{j=1}^{r} X_j^2 \right).
\]
As ${\mathcal H}=1$ for Einstein trajectories,
it follows that $dt = \frac{ds}{{\rm tr}L}$.

Note that we have
\begin{equation}  \label{dmetric}
\frac{\sqrt{d_i}}{\tr L} \frac{X_i}{Y_i^2} = \dot{g_i} g_i,
\end{equation}
which is consistent with our formula in the soliton case.

\medskip

As in the soliton case, we can restrict to the subsystem
obtained by omitting the equation for $Y_1$, and deduce that the
flow is defined for all $s$ since ${\mathcal L}=0$ is compact
for the subsystem. We have $X_i,Y_i >0$ along our trajectories as before.
As ${\mathcal L}=0$, we in fact  have $0 < X_i < 1$ for all $i$.
Note also that the variety $\{{\mathcal H}=1, {\mathcal L}=0 \}$ is smooth.

\begin{lemma}
The Ricci-flat metrics corresponding to our trajectories are complete.
\end{lemma}

\begin{proof}
We have
\[
dt = \exp \left(\int_{s^*}^{s} {\mathcal G} \right) \; ds
\]
where ${\mathcal G} = \sum_i X_i^2 \geq \frac{1}{n} {\mathcal H}^2$ by Cauchy-Schwartz.
Since ${\mathcal H} = 1$ along Einstein trajectories, it follows that
$t$ tends to $\infty$ as $s$ does, proving completeness.
\end{proof}

In order to examine the long time behaviour of the Ricci-flat trajectories,
we need to use a modified form of the Lyapunov function ${\mathscr F}_0$ for the flow
discussed in Corollary \ref{steady-F}. Writing ${\mathscr F}_0$ in terms of
the variables $X_i, Y_i$ (cf (\ref{def-Xi}) and (\ref{def-Yi})) we get
$$ {\mathscr F}_0 = \left(\sum_{i=1}^r \, X_i^2 + \sum_{i=1}^r \, \lambda_i Y_i^2
    - \frac{{\mathcal H}^2}{n} \right) \prod_{i=1}^r  \,
    \left(\frac{\sqrt{d_i}}{Y_i} \frac{1}{\xi} \right)^{-\frac{2d_i}{n}}.  $$
Taking into account the conditions ${\mathcal L}=0, {\mathcal H}=1$ and the
fact that $X_1$ plays a special role in the subsystem, we consider the
following modified Lyapunov function with domain
${\mathcal D}:=\{ {\mathcal L}=0, {\mathcal H}=1 \} \cap \{Y_i > 0 \,(i>1), \,|X_1 -1| < \sqrt{2} \} $:
\begin{equation} \label{mod-F}
\hat{\mathscr F} := \frac{1 - \frac{1}{n-1} (1 - X_1)^2}{\prod_{i=2}^{r}
       \left(\sqrt{\lambda_i Y_i} \right)^{\frac{2d_i}{n-1}}}
      =   \frac{\sum_{i=1}^{r} X_i^2 + \sum_{i=2}^{r} \lambda_i Y_i^2
       -\frac{1}{n-1} \left( {\mathcal H} - X_1 \right)^2}{ \prod_{i=2}^{r}
        \left(\sqrt{\lambda_i Y_i} \right)^{\frac{2d_i}{n-1}}}.
\end{equation}
Note that $\hat{\mathscr F}$ is positive along our trajectories as $0 < X_1 < 1$.

\begin{lemma} \label{F-decr}
$\hat{\mathscr F}$ is non-increasing along the trajectories of the flow lying in $\mathcal D$.
\end{lemma}
\begin{proof}
After some algebra we find
\[
\frac{1}{2} \frac{\hat{\mathscr F}^{\prime}}{\hat{\mathscr F}}
= \frac{X_1(1-X_1)({\mathcal G}-1) +(n-2 + 2X_1 - X_1^2)(\frac{1-X_1}{n-1} - {\mathcal G})}
{n-2 + 2X_1 - X_1^2}
\]
where ${\mathcal G} = \sum_{i=1}^{r} X_i^2$ as usual.
For our trajectories the denominator is positive. The numerator
may be rewritten as
\[
\frac{1-X_1}{n-1}\, (n-2 - (n-3)X_1 - X_1^2) + {\mathcal G}(-X_1 + 2-n),
\]
in which the term multiplying $\mathcal G$ is negative.
Now, using Cauchy-Schwartz and ${\mathcal H} =1$, we have the inequality
\[
 {\mathcal G}  \geq X_1^2 + \frac{1}{n-1} \left( \sum_{i=2}^{r} \sqrt{d_i} X_i \right)^2
= X_1^2 + \frac{(1-X_1)^2}{n-1}.
\]
Substituting into the above expression for the numerator, we find
after simplification that the numerator is $\leq X_1^2 (2-n - X_1)$
which is $\leq 0$.
\end{proof}

\begin{rmk} \label{Fstat}
We have $\hat{\mathscr F}^{\prime} =0$ iff $X_1=0$ and we have equality in
Cauchy-Schwartz, that is, when $X_i = \frac{\sqrt{d_i}}{n-1}$ for $i \geq 2$.
\end{rmk}

\begin{lemma} \label{Fmin}
The function $\hat{\mathscr F}$ has a unique critical point in $\mathcal D$
which is the global minimum point.
\end{lemma}

\begin{proof}
We use the second expression of $\hat{\mathscr F}$ in (\ref{mod-F}). By Cauchy-Schwartz
and $X_1 \geq 0$  the numerator is at least $\sum_{i=2}^{r} \lambda_i Y_i^2$.
Next, using similar calculations to those in Prop 4.10 of \cite{DHW} we find that
$\hat{\mathscr F} \geq (n-1) \prod_{i=2}^{r} d_i^{-\frac{d_i}{n-1}}$ in $\mathcal D$.
Equality holds exactly at the point $E$ whose coordinates are given by
\[
X_1=0, \; \;  X_i =\frac{\sqrt{d_i}}{n-1}, \,\,
Y_i = \sqrt{\frac{n-2}{\lambda_i}}\, X_i  \; \; : \;\; (i=2, \ldots r) ;
\]
in fact it is easy to check that $E$ is the unique critical point of $\hat{\mathscr F}$
in $\mathcal D$.
\end{proof}

We can now use $\hat{\mathscr F}$ to analyse the long-time behaviour of the flow.
\begin{thm} \label{omega}
The $r-2$ parameter family of Ricci-flat trajectories all converge to $E$ as $s$
tends to infinity.
\end{thm}

\begin{proof}
As usual we work with the subsystem omitting $Y_1$. As $\{{\mathcal L}=0,
{\mathcal H}=1 \}$ is now compact,  for each trajectory $\gamma$
we have an $\omega$-limit set $\Omega$ lying in the level set
$\hat{\mathscr F} = \mu$ where $\mu$ is the infimum of $\hat{\mathscr F}$ along the
trajectory. Notice that $\mu > 0$ by Lemma \ref{Fmin}, and from the expression
of $\hat{\mathscr F}$ none of the $Y_i$ coordinates of a point in $\Omega$ can be zero.
Hence $\Omega$ lies in $\mathcal D$. As $\Omega$ is flow-invariant, Remark \ref{Fstat} shows
that on $\Omega$ we have $X_1 =0$ and $X_i =\frac{\sqrt{d_i}}{n-1}$ for $i \geq 2$.
In particular, $Y_i^{\prime}$ must vanish. Again by flow-invariance,  we also need
$X_i^{\prime}$ to vanish, and this now forces $\Omega$ to be $\{ E \}$,
as all other stationary points in ${\mathcal L} = 0$ do not lie in ${\mathcal D}$.

So $\mu = \hat{\mathscr F}(E)$, the global minimum of $\hat{\mathscr F}$
in $\mathcal D$. Now let $\epsilon > 0$ be sufficiently small so that the
$\epsilon$-ball around $E$ in ${\mathcal L}=0, {\mathcal H}=1$ is contained
in the region where all $Y_i >0$. Recall that ${\mathcal L}=0, {\mathcal H}=1$
is smooth at $E$. From Lemma \ref{Fmin} we know $E$ is the unique point where
$\mu$ is attained. Therefore the minimum of $\hat{\mathscr F}$ on the $\epsilon$-sphere
around $E$ is $\mu + \delta$ for some $\delta>0$. As $E$ is the
$\omega$-limit set, there exists a time $s_*$ where the trajectory
lies in the open $\epsilon$-ball and $\hat{\mathscr F} (\gamma(s_*)) < \mu + (\delta/2)$.
Now by monotonicity of $\hat{\mathscr F}$ the trajectory can never pass back through the
$\epsilon$-sphere, so is trapped for all later time in the $\epsilon$-ball.
Hence the trajectory converges to $E$.
\end{proof}

\begin{rmk}
One can give an alternative proof of B\"ohm's existence result for complete Ricci-flat metrics
on multiply warped products with $d_1 > 1$ along the above lines by using instead the Lyapunov
function
\[
{\sf F} = \prod_{j=1}^{r} \, Y_j^{-\frac{2d_j}{n}}.
\]
$\sf F$ is again positive and non-increasing along the trajectories of the flow on the
locus where ${\mathcal H} =1, {\mathcal L} =0, X_i, Y_i >0 \, (i=1, ..., r)$. In this set, ${\sf F}$ has a
unique critical point, which is a global minimum, whose coordinates are
$X_i = \frac{\sqrt{d_i}}{n},\, Y_i = \sqrt{\frac{n-1}{\lambda_i}} X_i$. This point
corresponds geometrically to the Ricci flat cone on the product Einstein metric
of $S^{d_1} \times M_2 \times \cdots \times M_r$. An account of this alternative proof
can be found in the McMaster M. Sc. thesis of Cong Zhou.
\end{rmk}

We now consider the asymptotics of the complete Ricci-flat metrics we have constructed.
Note that ${\mathcal G} = \sum_{i=1}^{r} X_i^2$ equals $\frac{1}{n-1}$ at $E$. So we can choose
a sufficiently small positive $\delta$ and $s_1 > s^*$ such that $|{\mathcal G} - \frac{1}{n-1}| < \delta$
for all $s \geq s_1$. Equation (\ref{stE}) then gives us estimates
\[
\frac{\rho_0 (n-1)}{1 - \delta(n-1)} \, \left(e^{(\frac{1}{n-1} - \delta)(s-s_1)}  -1 \right) < t-t_1 <
\frac{\rho_0 (n-1)}{1 + \delta(n-1)} \, \left(e^{(\frac{1}{n-1} + \delta)(s-s_1)}  -1 \right)
\]
where $\rho_0$ is the constant $\exp (\int_{s^*}^{s_1} {\mathcal G})$, and $t_1$ corresponds to $s_1$ via (\ref{stE}).

\begin{lemma} \label{asymp-g1}
The function $g_1(t)^2$ is increasing and bounded from above and
hence converges to a positive constant as $t$ tends to infinity.
\end{lemma}

\begin{proof} That $g_1(t)^2$ is increasing follows from the formula (\ref{dmetric}).
Integrating this we obtain
$$ \frac{1}{2} (g_1(t)^2 - g_1(t_1)^2) = \int_{s_1}^{s} \, \frac{X_1}{Y_1^2}
        \, \exp\left(2 \int_{s^*}^{\sigma} \, \sum_j \, X_j^2   \right) d \sigma.$$
We shall estimate the integral by estimating $X_1, 1/Y_1^2$ and the exponential
separately.

The equation for $X_1$ implies that $ X_1^{\prime} \leq X_1 \left( -\frac{n-2}{n-1} + \delta  \right),$
which yields upon integration
$$ X_1(s) \leq X_1(s_1) \exp \left(-\left(\frac{n-2}{n-1}-\delta \right)(s-s_1)  \right). $$
The equation for $Y_1$ gives
$$ (\log Y_1)^{\prime} = \sum_{j=2}^{r} X_j^2 + X_1^2 - X_1 \geq \frac{(1-X_1)^2}{n-1} + X_1^2 - X_1
      = \frac{1}{n-1} - \left(\frac{n+1}{n-1}\right) X_1 + \frac{n}{n-1} X_1^2,$$
where we have used the Cauchy-Schwartz inequality and ${\mathcal H} =1$. Since $X_1$ tends to $0$
as $s$ tends to infinity, we may assume that $s_1$ has been also chosen so that the absolute value
of the terms involving $X_1$ in the above is less than $\delta$. Integration of the inequality then gives
$$ \frac{1}{Y_1(s)^2} < \frac{1}{Y_1(s_1)^2} \exp\left(-2(s-s_1)\left(\frac{1}{n-1} -\delta \right)\right). $$

Finally,
$$ \exp \int_{s^*}^{\sigma} \sum_j \,X_j^2  = \left(\exp \int_{s^*}^{s_1} {\mathcal G} \right)
       \left( \exp \int_{s_1}^{\sigma} {\mathcal G} \right) \,\leq
       \,\rho_0 \, \exp \left((\sigma -s_1)\left(\frac{1}{n-1}+ \delta \right)\right),$$
where $\rho_0 = \exp(\int_{s^*}^{s_1} {\mathcal G})$ (which depends in particular on the choice of $\delta$).

Now combining the three inequalities we get
$$ \frac{1}{2} (g_1(t)^2 - g_1(t_1)^2) < \rho_0^2\, \left( \frac{X_1(s_1)}{Y_1(s_1)^2}\right)
         \int_{s_1}^{s} \exp \left(\left(-\frac{n-2}{n-1} + 5\delta \right)(\sigma - s_1)\right) d\sigma. $$
As $\delta$ can be chosen arbitrarily small, it follows that $g_1(t)^2$ is bounded above for all $t$.
\end{proof}

Similarly, arguing as in the soliton case, and using the fact
that $\lim_{s \rightarrow \infty} \frac{X_i}{Y_i^2} = \frac{X_i}{Y_i^2}(E)
=\frac{n-1}{n-2} \frac{\lambda_i}{\sqrt{d_i}}$ for $i \geq 2$, we obtain
estimates for $g_i^2(t)- g_i^{2}(t_1)$ for all $t > t_1$. These imply that asymptotically
\begin{equation} \label{cone-asymp}
  c_1 t^{2-\epsilon_0} < g_i(t)^2 < c_2 t^{2+\epsilon_0}
\end{equation}
for arbitrarily small $\epsilon_0 > 0$ and positive constants $c_1, c_2$
depending on $\epsilon_0$ and $\delta$.

\begin{rmk} The asymptotics obtained above are an analogue of those of the
Riemannian Schwarzschild metric, which is the case where $r=2$ and $M_2 = S^2$
(with the constant curvature $1$ metric).

We note also that $Y_1$ tends to infinity (exponentially fast in $s$). So in the
full phase space, the trajectories of the Ricci-flat metrics are indeed unbounded.
\end{rmk}

We have therefore proved

\begin{thm} \label{thm-ricci-flat}
Let $M_2, \ldots, M_r$ be closed Einstein manifolds with positive scalar
curvature. There is an $r-2$ parameter family of non-homothetic complete
smooth Ricci flat metrics of form $($\ref{metric}$)$ on the trivial rank $2$
vector bundle over $M_2 \times \ldots \times M_r$.  Asymptotically, $g_1(t)$
tends to a positive constant and $g_i(t)^2, i>1$  are approximately quadratic
in the sense of $($\ref{cone-asymp}$)$.  \qed
\end{thm}

\begin{rmk} \label{superpot}
The cohomogeneity one Einstein equations can be viewed as the flow on the zero energy
hypersurface of a certain Hamiltonian ${\sf H}$ constructed in \cite{DW4}, \S 1. Recall
that a superpotential $\Phi$ of ${\sf H}$ is a $C^2$ function on the full momentum phase
space such that the equation ${\sf H}(q, d\Phi_q) = 0$ holds. It is shown in \cite{DW5} that
a superpotential automatically gives rise to a first order subsystem of the cohomogeneity
one Einstein equations. In the Ricci-flat case, physicists have frequently been able to show
that solutions to the first order subsystem represented metrics with special holonomy
(see, e.g., \cite{CGLP1} and \cite{CGLP2}).

On the other hand, there are examples of superpotentials which are not associated with
special holonomy, but nevertheless are related to some degree of integrability of
the Ricci-flat equations. We mention here case (1) of Theorem 6.3 and Examples 8.2 and 8.3
in \cite{DW4}. The hypersurfaces in these examples are respectively the product
of one, two, or three Einstein manifolds with positive scalar curvature. The dimensions
of the factors in the latter two cases are restricted, i.e., up to permutation, they
must be $(6,3), (8,2), (5,5)$ and $(3, 3, 3), (4, 4, 2), (6, 3, 2)$ respectively.
With an appropriately chosen sphere as one of the Einstein factors, there are {\em explicit}
solutions of the first order subsystem which are complete smooth Ricci-flat metrics,
and these must occur among the B\"ohm metrics since $d_1 > 1$ is satisfied.

It turns out that if we take a product of the above examples with a circle, we obtain
hypersurfaces with superpotentials as well (cf Remark 2.8 of \cite{DW4}, which applies
also to the null case). However, there are two differences. First, the convex polytopes
associated to the scalar curvature function and superpotentials are no longer of maximal
dimension. This explains why these examples did not occur in the classifications in
\cite{DW6} and \cite{DW7}.  Second, since the circle can be taken to be the collapsing factor,
none of the positive Einstein factors need to be spheres anymore in order to obtain
complete smooth solutions of the first order subsystem. These Ricci-flat metrics
must occur among those constructed in this section. Case (1) of Theorem 6.3 now becomes
the $r=2$ case, which is known to be explicitly integrable (cf \cite{BB}).
\end{rmk}

\medskip

The topology of the underlying manifolds where we have constructed steady soliton
and Ricci-flat structures is also very interesting. We shall consider here the $r=2$
case briefly.

For the Einstein factor $M_2$ we can take the Kervaire sphere $\Sigma$
of dimension $q=4m+1$ with $m \neq 0, 1, 3, 7, 15$ and possibly $31$, or one of the
homotopy spheres in dimension $7, 11, 15$ which bound a parallelizable manifold.
The solution of the Arf-Kervaire invariant problem by Hill, Hopkins, and Ravenel \cite{HHR} implies that
in the above dimensions the Kervaire sphere is not diffeomorphic to the standard sphere.
At the same time, the work of Boyer, Galicki, and Kollar \cite{BGK1}, \cite{BGK2} provides
continuous families of Sasakian Einstein metrics on these homotopy spheres.
On the other hand, it is known (Theorem 1 in \cite{KS}) that for a non-standard homotopy $q$-sphere
$\Sigma$ that bounds a parallelizable manifold, $\R^2 \times \Sigma$ is not diffeomorphic to
$\R^2 \times S^q$. Therefore, Theorems \ref{thm-ricci-flat} and \ref{mainthm} imply

\begin{cor} \label{diffstr}
In dimensions $9, 13, 17$ and all dimensions $4m+3$ with $m \neq 0, 1, 3, 7, 15, 31$
there exist pairs of homeomorphic but not diffeomorphic manifolds both of which admit
a complete Ricci-flat metric. The same conclusion holds for non-Einstein, complete,
steady gradient Ricci soliton structures.
\end{cor}

We are not aware of examples of this type in the literature, although they presumably
occur among Calabi-Yau manifolds. The soliton case is somewhat
surprising in view of the comparatively greater rigidity of the soliton equations.

\section{\bf Numerical examples}

In this section we shall produce some numerical solutions of the equations
(\ref{TT})-(\ref{TS}). We begin with the case of triple warped products
for which we obtained an existence proof in \S 3. The purpose here is to
see how quickly the characteristic asymptotics of steady gradient Ricci solitons
given in \S 1 and \S 3 exhibit themselves. We then produce numerical solutions in two
cases for which we do not yet have an analytic existence proof. The principal orbits
in these cases are the twistor space of quaternionic projective space (viewed
as a quaternionic K\"ahler manifold) and the total space of the corresponding
${\rm Sp}(1)$ bundle. For these examples, the numerics seem to indicate strongly
that complete steady soliton structures do exist.

The procedure we use is the same as that described in \cite{DHW}, \S 5. First,
a series solution is generated for these equations in a neighbourhood of
the singular orbit satisfying the appropriate smoothness conditions. We then
truncate the series and use the values of the resulting functions at some small
$t_0 > 0$ as initial values to generate solutions of the equations for $t > t_0$
via a fourth order Runge-Kutta scheme. This procedure is necessary because
the cohomogeneity one gradient  Ricci soliton equation has an irregular singular point
at $t=0$, the position of the singular orbit. Since non-trivial steady solitons are
necessarily non-compact, we rely on checks against the expected asymptotic behaviour
(cf \S 1) to ascertain that the numerical solutions indeed point to an actual solution.

\medskip

\noindent{\bf A. Triple Warped Product Case:}

\medskip

Taking $r=3$ for a metric of the form (\ref{metric}), the equations (\ref{TT}), (\ref{SS})
and the conservation law (\ref{cons1}) become the system

\begin{eqnarray*} \label{tripleEQ}
\dot{z_1} & = & z_2  \\
\dot{z_2} & = & -d_2 \frac{z_2 z_4}{z_3} - d_3 \frac{z_2 z_6}{z_5} + z_2 z_8 + \frac{\epsilon}{2}z_1 \\
\dot{z_3} & = & z_4 \\
\dot{z_4} & = & -(d_2-1)\frac{z_4^2}{z_3} - \frac{z_2 z_4}{z_1} -d_3 \frac{z_4 z_6}{z_5} + z_4 z_8
        + \frac{d_2 -1}{z_3} + \frac{\epsilon}{2}z_3 \\
\dot{z_5} & = & z_6  \\
\dot{z_6} & = &  -(d_3-1)\frac{z_6^2}{z_5} - \frac{z_2 z_6}{z_1} -d_2 \frac{z_4 z_6}{z_3} + z_6 z_8
        + \frac{d_3 -1}{z_5} + \frac{\epsilon}{2}z_5 \\
\dot{z_7} & = & z_8 \\
\dot{z_8} & = & -z_8 \left( \frac{z_2}{z_1} + d_2 \frac{z_4}{z_3} + d_3 \frac{z_6}{z_5} \right)
            + z_8^2 + \epsilon z_7 + C
\end{eqnarray*}
where $(z_1, \ldots, z_8) :=(g_1, \dot{g_1}, g_2, \dot{g_2}, g_3, \dot{g_3}, u, \dot{u})$ and we have
chosen $\lambda_2 = d_2 -1, \lambda_3 = d_3 -2$.

For steady solitons, we set $\epsilon = 0$ and we may rule out homothetic solitons by
choosing $C$ to be $-1$. Recall from \cite{DHW}, \S 2 that the smoothness conditions
together with the conservation law (\ref{cons1}) yield the general relation
\begin{equation} \label{initial-u}
  (d_1 + 1)\; \ddot{u}(0) = C + \epsilon u.
\end{equation}
Since $d_1 = 1$ and $\epsilon = 0$, we can further set $u(0)= 0$, and (\ref{initial-u}) becomes
 $\ddot{u}(0) = -\frac{1}{2}$. The initial values of $(z_1, \ldots, z_8)$ are then
 $(0, 1, a, 0, b, 0, 0, 0)$ with $a, b > 0$. (Recall that Theorem \ref{mainthm} gives a
 $2$-parameter family of solutions.) The stepsize in the Runge-Kutta scheme is $0.001$,
 so that the accumulated error is of the order $0.001^4$.

\medskip

\noindent{\bf Example 1.} \: We take $d_2=2$ and $d_3=3$, so the cohomogeneity
one manifold has dimension $7$. We include below two sketches of the solution with
initial values $a=6$ and $b=3$. Figure $1$ shows the functions $g_i$ and $u$ over
the range $0 \leq t \leq 10$, and Figure 2 does it for the range $0 \leq t \leq 500$.
Notice that the slope of the potential $u$ quickly becomes approximately equal
to $-1=-\sqrt{-C}$, cf Prop \ref{steady-MC}. The function $g_1$ also quickly
becomes approximately  constant. The limiting value can be made close to $0$ by
 making $a, b$ small. It also appears that regardless of the values of $a, b$,
 this limiting value is bounded above by $2$.

\begin{figure}[hbtp]
\includegraphics[width=15cm,height=10cm]{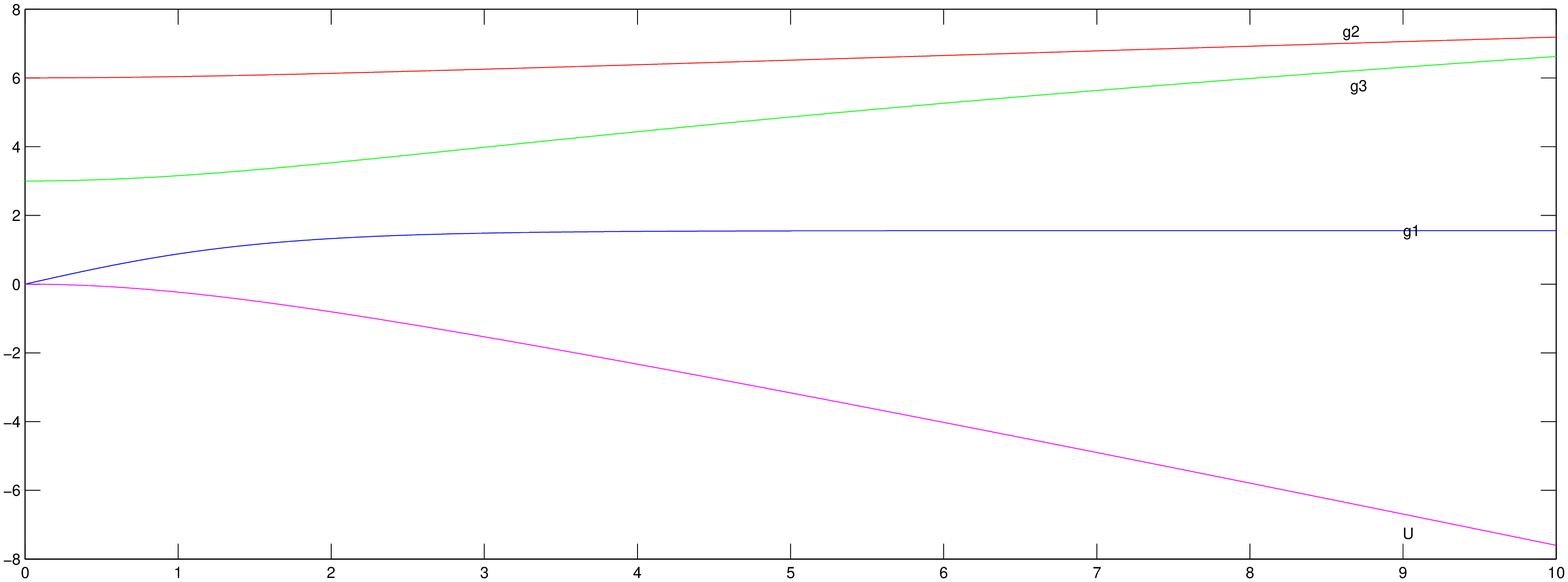}
\caption{}
\end{figure}

\pagebreak

\begin{figure}[hbtp]
\includegraphics[width=15cm,height=8cm]{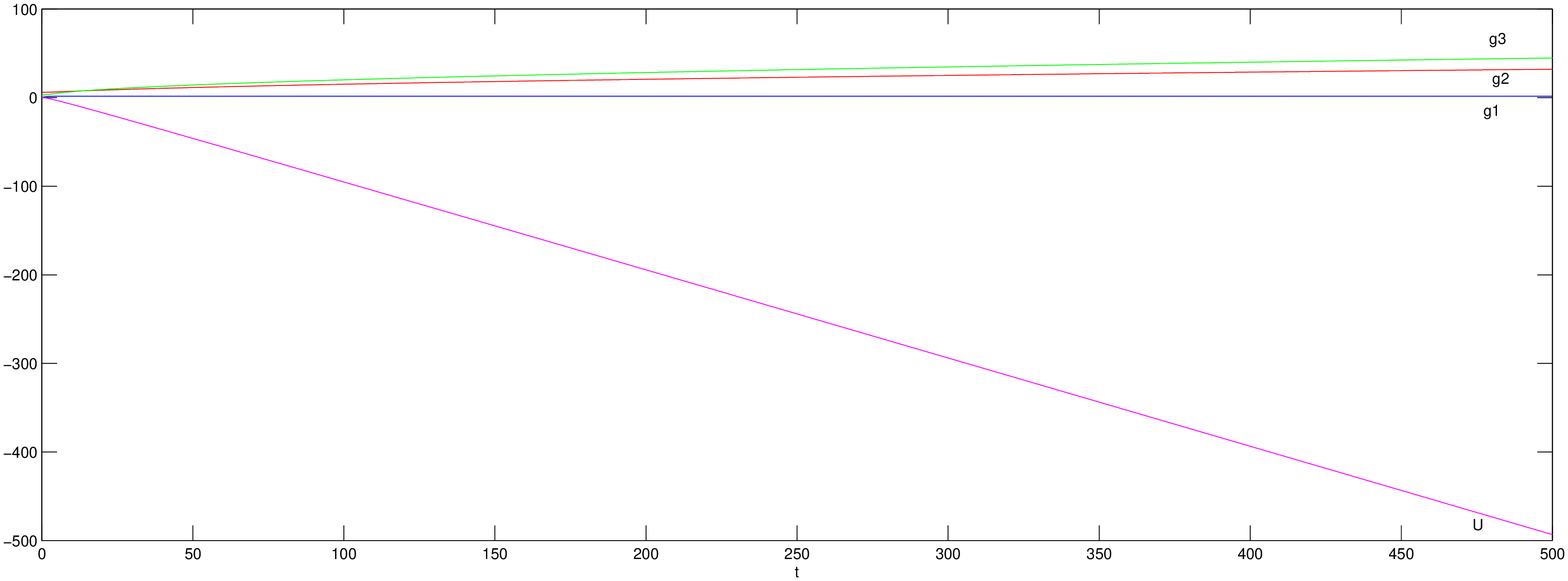}
\caption{}
\end{figure}

We also plot the quantities $ {\tilde X}_i := \frac{X_i}{\sqrt{d_i}},
\;\;\;  {\tilde Y}_i:=\frac{Y_i}{\sqrt{d_i}}$ against $t$ respectively
in Figures 3 and 4. Recall that the proof of Theorem \ref{asymptotics}
shows that ${\tilde Y}_1$ tends to a positive constant while the remaining
quantities tend to zero.

\begin{figure}[hbtp]
\includegraphics[width=15cm,height=9cm]{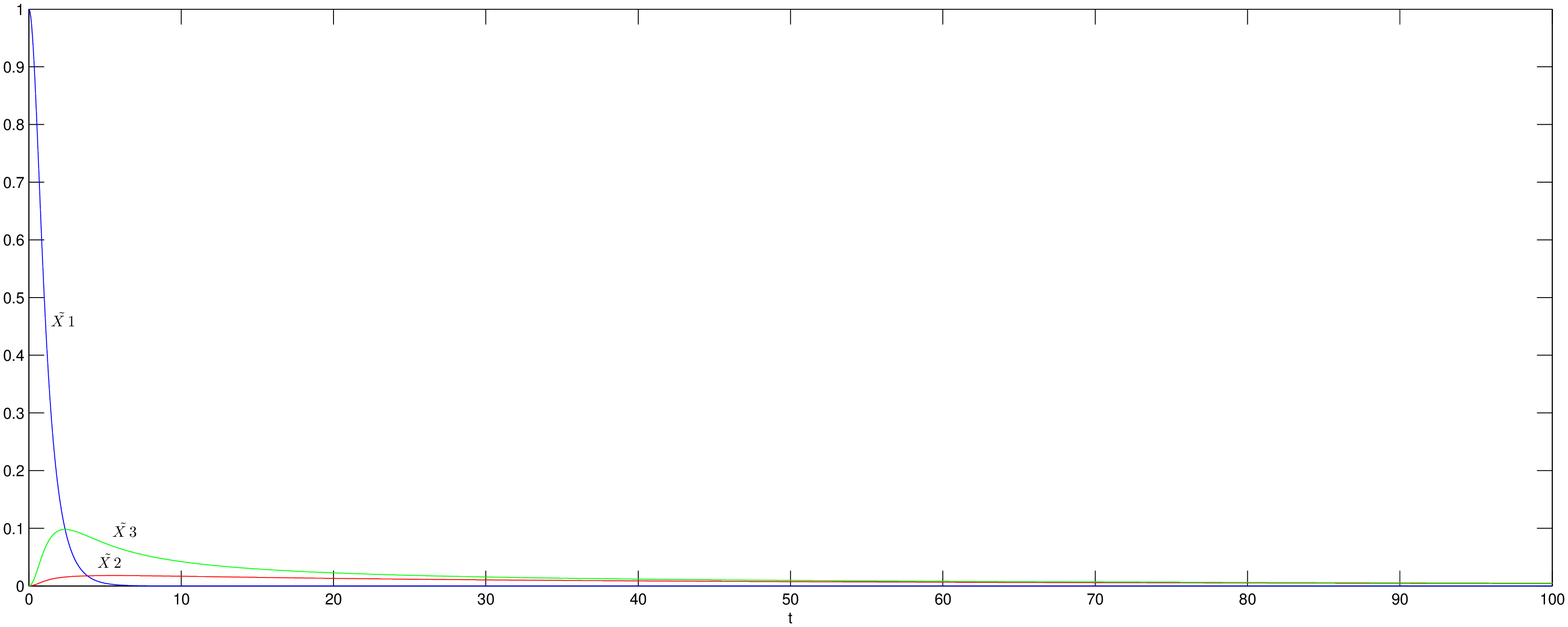}
\caption{}
\end{figure}

\begin{figure}[hbtp]
\includegraphics[width=15cm,height=10cm]{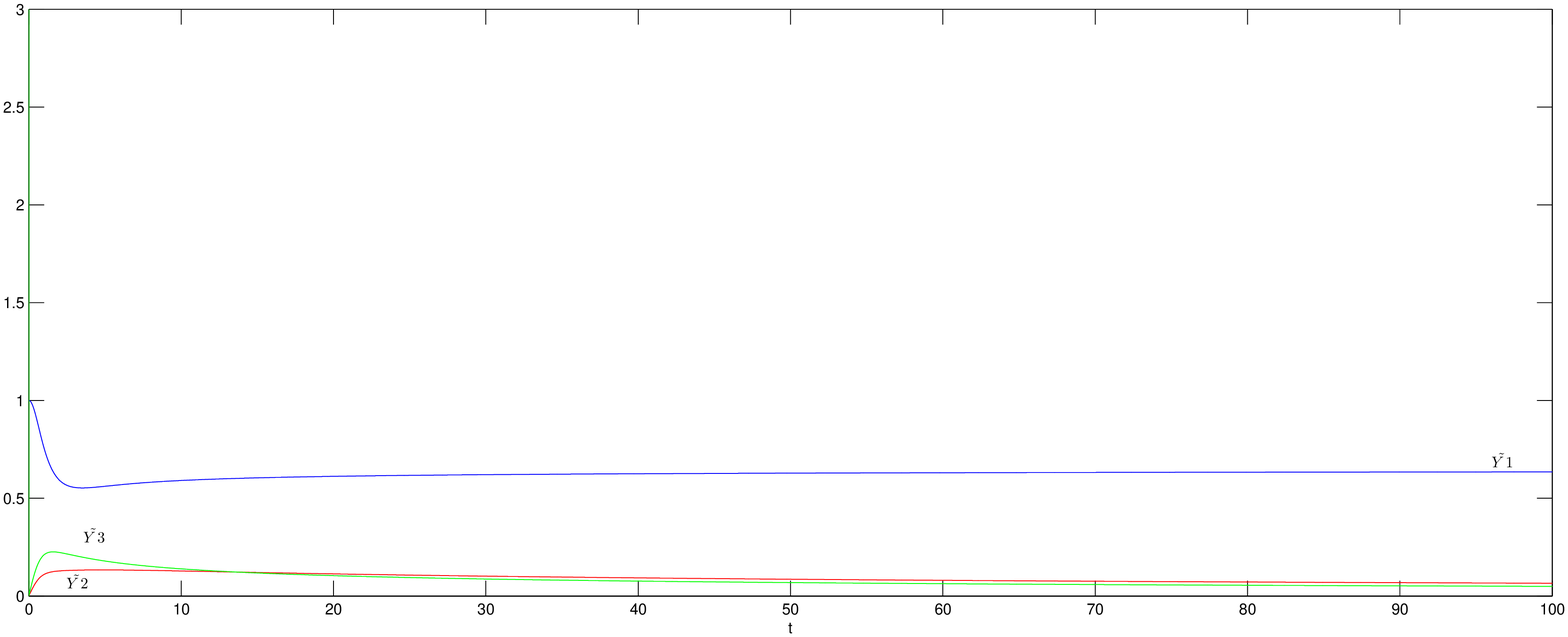}
\caption{}
\end{figure}

\pagebreak

\noindent{\bf B. Some Two-summands Cases:}

\medskip

We consider next a principal orbit $G/K$ whose isotropy representation consists of two
inequivalent ${\rm Ad}(K)$-invariant irreducible real summands. We assume that
$K \subset H \subset G$ where $H, K$ are closed subgroups of the compact Lie group $G$
such that $H/K$ is a sphere. We choose a $G$-invariant background metric $\mathsf b$ on $G/K$
such that it induces the constant curvature $1$ metric on $H/K$. Our cohomogeneity one
manifold $M$ is then the vector bundle $G \times_{H} \R^{d_{1}+1}$ where $H/K \subset \R^{d_{1}+1}$
is regarded as the unit sphere.

Let $\g = \kf \oplus \p$ be an ${\rm Ad}(K)$-invariant decomposition of the Lie algebra
of $G$. Then $\p$ is identified with the tangent space of $G/K$ at the base point.
We can further decompose $\p$ as $\p_1 \oplus \p_2$, where $\p_i$ are $\mathsf b$-orthogonal
irreducible ${\rm Ad}(K)$-representations of dimensions $d_i$. Note that $\p_1$ is
the tangent space to the sphere $H/K$ at the base point, and $\p_2$ is the tangent
space of the singular orbit $G/H$ at the base point. Our metrics of cohomogeneity one
then take the form
$$ \bar{g} = dt^2 + g_1(t)^2 \; {\mathsf b}|\p_1 + g_2(t)^2 \; {\mathsf b}|\p_2.$$

As is well-known, the Ricci operator of the metric $g(t)$ on $G/K$ has corresponding
components
$$ r_1 = \frac{A_1}{d_1} \frac{1}{g_1^2} + \frac{A_3}{d_1} \frac{g_1^2}{g_2^4} $$
$$ r_2 = \frac{A_2}{d_2} \frac{1}{g_2^2} - \frac{2A_3}{d_2} \frac{g_1^2}{g_2^4}. $$
As a result of our choice of $\mathsf b$, one has $A_1=d_1(d_1-1)$, and the constants $A_2, A_3$
are analogous (but not always equal) to the quantities $c_Q$ and $\|{\mathcal A}\|$ in \cite{DHW}.

Taking $(z_1, \ldots, z_6):=(g_1, \dot{g_1}, g_2, \dot{g_2}, u, \dot{u})$ as before, the
soliton equations become
\begin{eqnarray*} \label{2summands}
\dot{z_1} & = & z_2  \\
\dot{z_2} & = & -(d_1 -1)\frac{z_2^2}{z_1} -d_2 \frac{z_2 z_4}{z_3} + z_2 z_6 + \frac{d_1 -1}{z_1}
            + \frac{A_3}{d_1} \frac{z_1^3}{z_3^4} + \frac{\epsilon}{2}z_1 \\
\dot{z_3} & = & z_4 \\
\dot{z_4} & = & -d_1 \frac{z_2 z_4}{z_1} -(d_2-1)\frac{z_4^2}{z_3} + z_4 z_6
        + \frac{A_2}{d_2} \frac{1}{z_3} -2\frac{A_3}{d_2} \frac{z_1^2}{z_3^3} + \frac{\epsilon}{2}z_3 \\
\dot{z_5} & = & z_6 \\
\dot{z_6} & = & -z_6 \left( d_1 \frac{z_2}{z_1} + d_2 \frac{z_4}{z_3} \right)
            + z_6^2 + \epsilon z_5 + C,
\end{eqnarray*}
with $\epsilon=0$ and $C=-1$.

\medskip

\noindent{\bf Example 2.} \: We set $G={\rm Sp}(m+1),  H={\rm Sp}(m) \times {\rm Sp}(1),$
and $K={\rm Sp}(m) \times {\rm U}(1)$. The principal orbit $G/K$ is diffeomorphic to $\C{\PP}^{2m+1}$
and the singular orbit $G/H$ is $\HH{\PP}^m$. So $d_1=2, d_2 = 4m$, and $A_2 = 2m(m+2), A_3=\frac{m}{2}$
(where $\mathsf b$ is given by $-2\tr(XY)$). As before we set $u(0)=0$ and (\ref{initial-u}) becomes
$\ddot{u}(0) = -\frac{1}{3}$. The initial values of $(z_1, \ldots, z_6)$ are given by
$(0, 1, \bar{h}, 0, 0, 0)$ where $\bar{h} > 0.$

We plot the functions $g_i$ and $u$ for the $m=1$ and $m=2$ cases respectively
in Figures $5$ and $6$. Recall that the dimensions of the cohomogeneity one
manifolds are respectively $7$ and $11$ for the $m=1$ and $m=2$ cases.
The initial value $\bar{h}$ is taken to be $6$ in both cases. Again the soliton
potential very quickly becomes approximately linear with slope close to $-1$.
The growth of the functions $g_i$ is consistent with being  like
$\sqrt{t}$ asymptotically.

\pagebreak

\begin{figure}[hbtp]
\includegraphics[width=15cm,height=9cm]{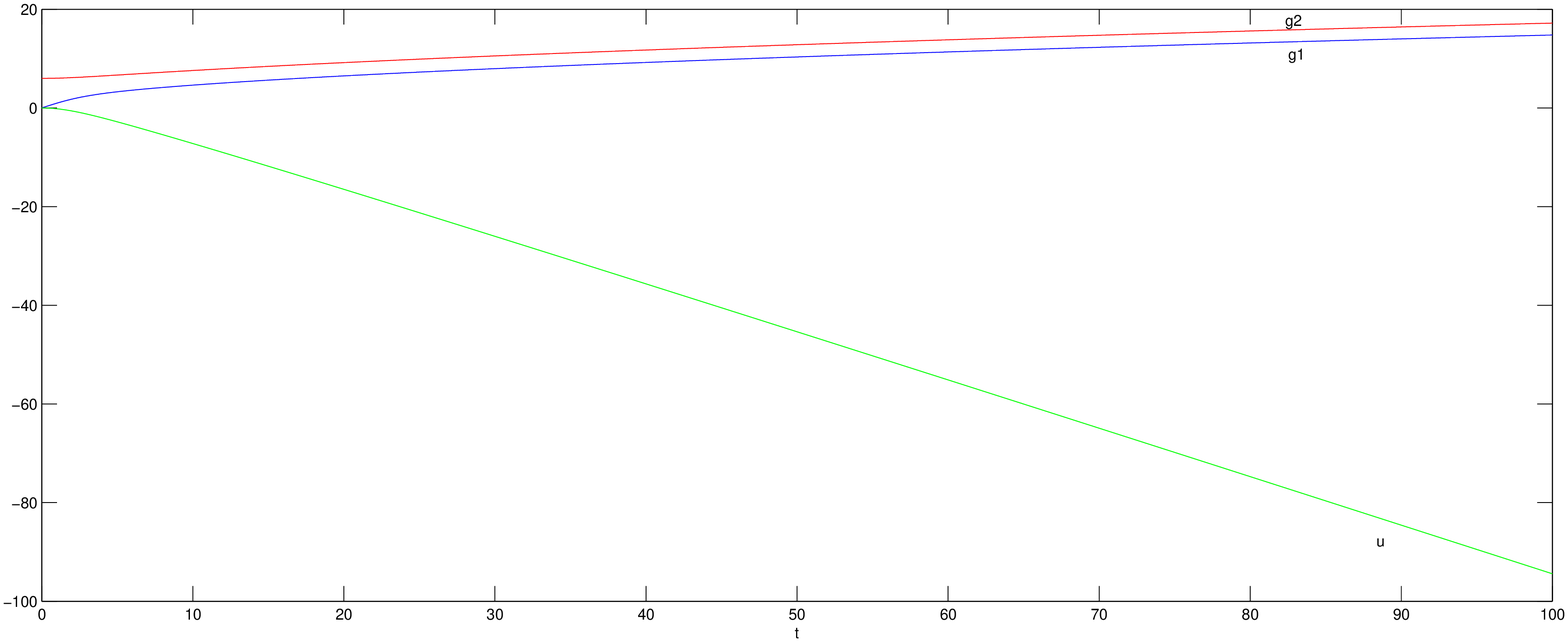}
\caption{}
\end{figure}

\begin{figure}[hbtp]
\includegraphics[width=15cm,height=9cm]{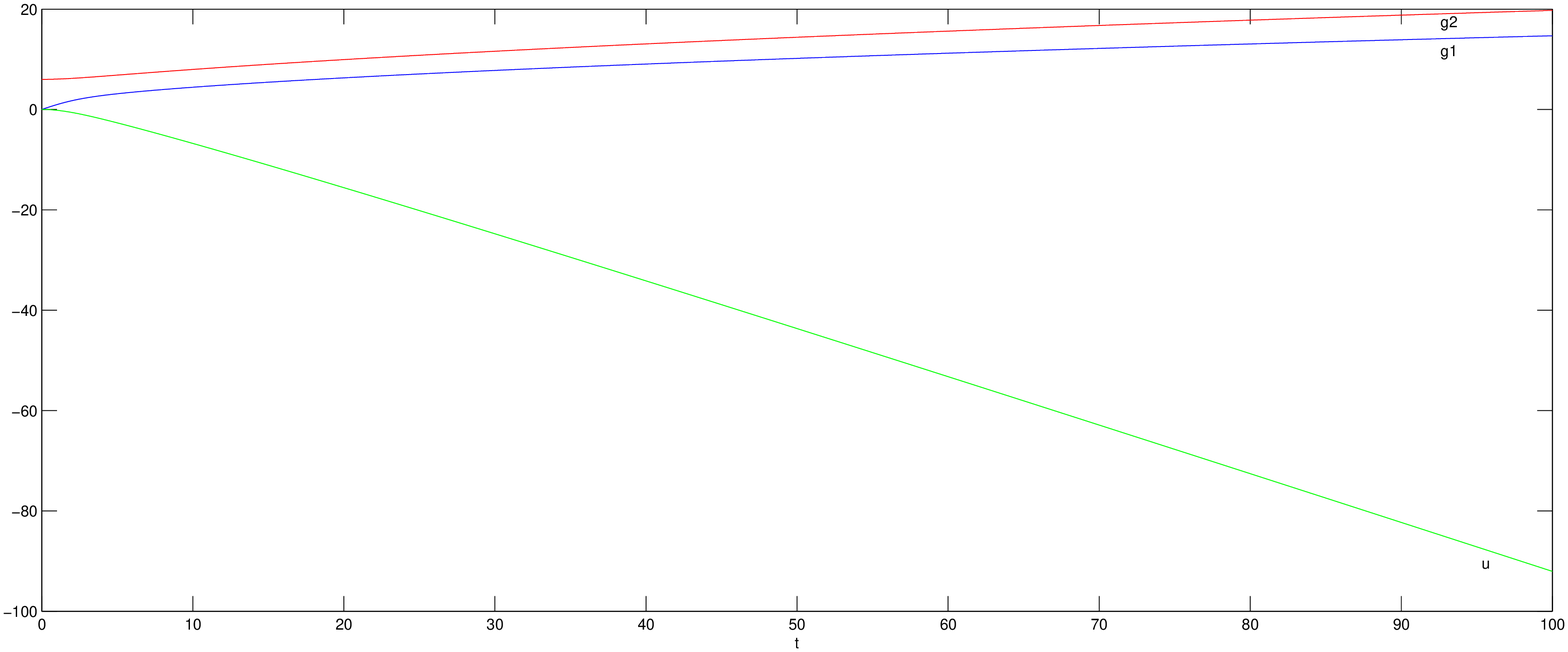}
\caption{}
\end{figure}

When the isotropy representation of the principal orbit is multiplicity free,
we can still define the new variables $s$ and $X_i, Y_i$ as in \S 2.
We again let ${\tilde X_i}= X_i/\sqrt{d_i}$ and ${\tilde Y_i}= Y_i/\sqrt{d_i}$.
Now we can plot the quantities ${\tilde X}_i, {\tilde Y}_i$ against $t$. Figure $7$
corresponds to the $m=1$ case and Figure $8$ to the $m=2$ case. All
these quantities appear to tend to zero, but the ${\tilde X}_i$ do so faster.
Note that since by Proposition \ref{steady-MC} the quantity $\xi$ tends to $1$,
the fact that the ${\tilde X}_i$ appear to tend to zero very likely indicates
that the shape operator $L$ also tends to $0$.

\begin{figure}[hbtp]
\includegraphics[width=17cm,height=9cm]{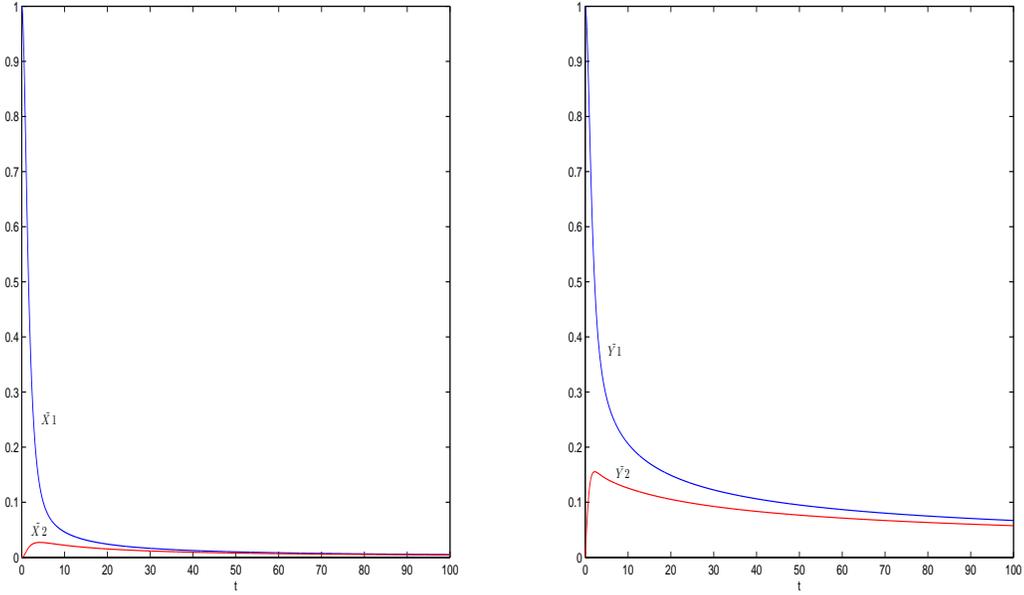}
\caption{${\tilde X_i}, {\tilde Y_i}$ for $m=1$}
\end{figure}

\begin{figure}[hbtp]
\includegraphics[width=17cm,height=9cm]{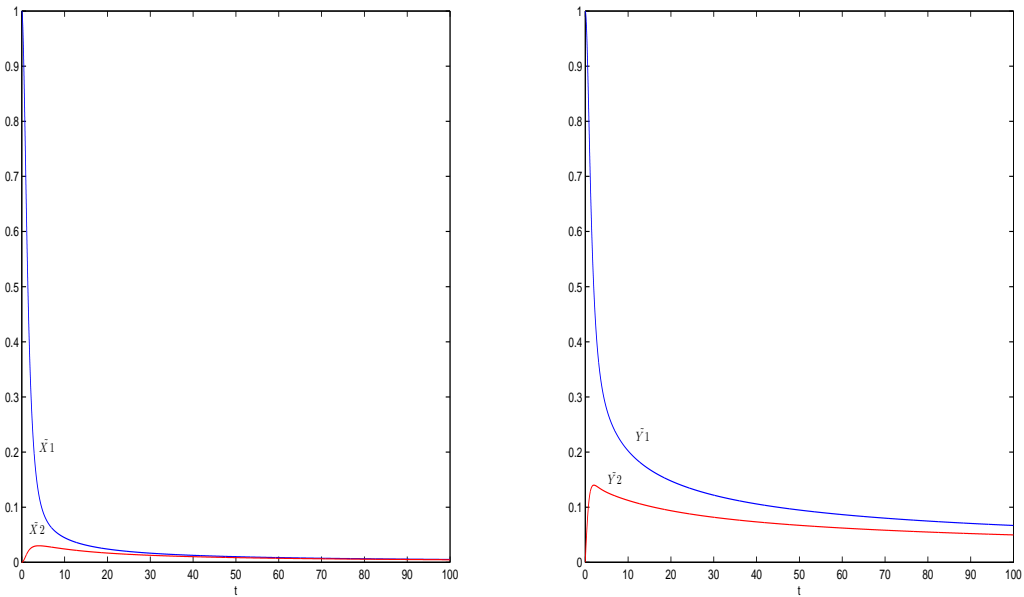}
\caption{${\tilde X_i}, {\tilde Y_i}$ for $m=2$}
\end{figure}

We can also plot the ratios ${\tilde X}_1/{\tilde X_2}=\frac{\dot{g_1}}{g_1} \frac{g_2}{\dot{g_2}}$
 and ${\tilde Y}_1/{\tilde Y}_2 = \frac{g_2}{g_1}$ against $t$. For the $m=1$ case this is
shown in Figure 9. The first ratio appears to tend to $1$, which likely indicates that the
principal curvatures of the principal orbits asymptotically become equal. It is not as
convincing (from considering other values of $m$) that the second ratio tends to
$1$, although it does appear to tend to some limit.

\begin{figure}[hbtp]
\begin{center}
\includegraphics[width=16cm,height=8cm]{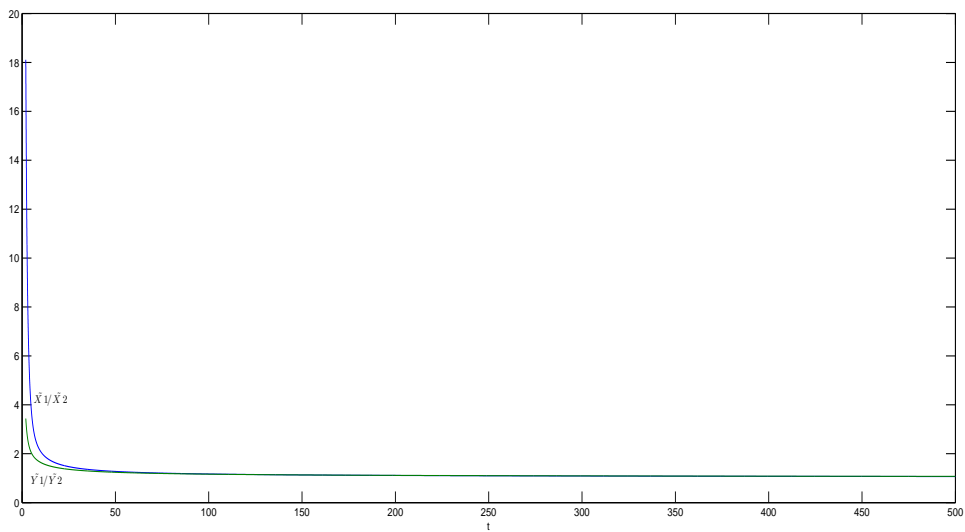}
\caption{${\tilde X_1}/{\tilde X_2}$ and ${\tilde Y_1}/{\tilde Y_2}$ for $m=1$}
\end{center}
\end{figure}

Similar numerical results hold for larger values of $m$.

\medskip

\noindent{\bf Example 3.} \: We set  $G={\rm Sp}(m+1) \times {\rm Sp}(1),
H={\rm Sp}(m) \times {\rm Sp}(1) \times {\rm Sp}(1),$ and $K={\rm Sp}(m) \times \Delta{\rm Sp}(1)$.
The principal orbit $G/K$ is diffeomorphic to $S^{4m+3}$ and the singular orbit $G/H$ is
again $\HH{\PP}^m$. So $d_1=3, d_2 = 4m$, and $A_2 = 4m(m+2), A_3=\frac{3m}{4}$
(where $\mathsf b$ is given by $-\tr(XY)$ on both the ${\rm Sp}(m+1)$ and ${\rm Sp}(1)$
factors). As before we set $u(0)=0$ and (\ref{initial-u}) becomes
$\ddot{u}(0) = -\frac{1}{4}$. The initial values of $(z_1, \ldots, z_6)$ are given by
$(0, 1, \bar{h}, 0, 0, 0)$ where $\bar{h} > 0.$

For this case we obtain graphs very similar to those in Example 2.

\medskip

Based on the last two examples, we would conjecture that on the
vector bundles $G \times_H \R^{d_1+1},$ where $(G, H,  K)$ are as above, there is a
$1$-parameter family of non-homothetic complete steady gradient Ricci solitons.
We hope to pursue this and related questions in subsequent work.


\begin{thebibliography}{bbbbbb}

\bibitem[ACGT]{ACGT} V. Apostolov, D. Calderbank, P. Gauduchon and C. T{\o}nnesen-Friedman,
     {\em Hamiltonian $2$-forms in K\"ahler geometry IV: weakly Bochner-flat
     K\"ahler manifolds}, Comm. Anal. Geom. {\bf 16}, (2008), 91-126.
\bibitem[BB]{BB} L. B\'erard Bergery, {\em Sur des nouvelles vari\'et\'es
    Riemanniennes d'Einstein}, Publications de l'Institut Elie Cartan,
     Nancy, (1982).
\bibitem[Be]{Be} A. Besse, {\em Einstein Manifolds}, Ergebnisse der
    Mathematik und ihrer Grenzgebiete, 3. Folge, Band 10, Springer-Verlag, (1987).
\bibitem[Bo1]{Bo1} C. B\"ohm, {\em Non-compact cohomogeneity one Einstein manifolds,}
      Bull. Soc. Math. France, {\bf 122}, (1999), 135-177.
\bibitem[Bo2]{Bo2} C. B\"ohm, {\em Inhomogeneous Einstein metrics on low-dimensional
       spheres and other low-dimensional spaces}, Invent. Math., {\bf 134}, (1998), no. 1, 145-176.
\bibitem[BGK1]{BGK1} C. Boyer, K. Galicki and J. Koll{\' a}r, {\em Einstein metrics on
       spheres}, Ann. Math., {\bf 162}, (2005), 557-580.
\bibitem[BGK2]{BGK2} C. Boyer, K. Galicki and J. Koll{\' a}r, {Einstein metrics on
    exotic spheres in dimension $7, 11, $ and $15$,} {\em Experiment. Math.,} {\bf 14},
    (2005), 59-64.
\bibitem[Br1]{Br1} S. Brendle, {\em Rotation symmetry of self-similar solutions to the
       Ricci flow}, to appear in Invent. Math., arXiv:math.DG/1202.1264.
\bibitem[Br2]{Br2} S. Brendle, {\em Rotation symmetry of solitons in higher dimensions},
        to appear in J. Diff. Geom., arXiv:math.DG/1203.0270.
\bibitem[Bry]{Bry} R. Bryant, unpublished work.
\bibitem[BryS]{BryS} R. Bryant and S.M. Salamon,  {\em On the construction of
    some complete metrics with exceptional holonomy}, Duke Math J., {\bf 58},
    (1989), 829-850.
\bibitem[Buz]{Buz} M. Buzano, {\em Initial value problem for cohomogeneity one
         gradient Ricci solitons}, J. Geom. Phys, {\bf 61}, (2011), 1033-44.
\bibitem[Cao]{Cao} H. D. Cao, {\em Existence of gradient Ricci solitons},
     Elliptic and Parabolic Methods in Geometry, A. K. Peters,
     (1996), 1-16.
\bibitem[Chen]{Chen} B. L. Chen, {\em Strong uniqueness of the Ricci flow},
       J. Diff. Geom., {\bf 82 }, (2009),363-382.
\bibitem[Cetc]{Cetc} B. Chow, S.C. Chu, D. Glickenstein, C. Guenther,
    J. Isenberg, T. Ivey, D. Knopf, P. Lu, F. Luo, and L. Nei,
    {\em The Ricci flow: techniques and applications Part I: geometric aspects},
    Mathematical Surveys and Monographs Vol. 135, American Math. Soc. (2007).
\bibitem[CGLP1]{CGLP1} M. Cveti{\u{c}}, G. Gibbons, H. L\"u and C. Pope,
    {\em Supersymmetric $M3$-branes and $G_2$ manifolds}, Nucl. Phys. B, {\bf 620},
     (2002), 3-28.
\bibitem[CGLP2]{CGLP2} M. Cveti{\v c},  G. Gibbons, H. L{\" u} and C. Pope,
     {\em New complete noncompact $Spin(7)$ manifolds}, Nucl. Phys. B
     {\bf 620}, (2002), 29-54.
\bibitem[DHW]{DHW} A. Dancer, S. Hall and M. Wang, {\em Cohomogeneity one
          shrinking Ricci solitons: an analytic and numerical study},
           Asian J. Math., {\bf 17}, (2013), no. 1, 33-61.
\bibitem[DW1]{DW1} A. Dancer and M. Wang, {\em On Ricci solitons
             of cohomogeneity one}, Ann.  Glob. Anal. Geom., {\bf 39}, (2011) 259-292.
\bibitem[DW2]{DW2} A. Dancer and M. Wang, {\em Some new examples of
      non-K\"ahler Ricci solitons}, Math. Res. Lett. {\bf 16}, (2009) 349-363.
\bibitem[DW3]{DW3} A. Dancer and M. Wang, {\em Non-K\"ahler expanding
           Ricci solitons}, IMRN, ({\bf 2009}), 1107-33.
\bibitem[DW4]{DW4}  A. Dancer and M. Wang, {\em The cohomogeneity one Einstein equations
           from the Hamiltonian viewpoint},  \ J. reine angew. Math., {\bf 524}, (2000), 97-128.
\bibitem[DW5]{DW5} A. Dancer and M. Wang, {\em Superpotentials and the cohomogeneity
    one Einstein equations}, Comm. Math. Phys., {\bf 260}, (2005), 75-115.
\bibitem[DW6]{DW6} A. Dancer and M. Wang, {\em Classification of superpotentials},
          Comm. Math. Phys., {\bf 284}, (2008), 583-647.
\bibitem[DW7]{DW7} A. Dancer and M. Wang, {\em Classifying superpotentials: three
     summands case}, J. Geom. Phys., {\bf 61}, (2011), 675-692.
\bibitem[FIK]{FIK} M. Feldman, T. Ilmanen, and D. Knopf, {\em Rotationally symmetric
      shrinking and expanding gradient K\"ahler-Ricci solitons}, J. Diff. Geom.,
       {\bf 65}, (2003), 169-209.
\bibitem[FR]{FR} M. Fern{\'a}ndez-L{\'o}pez and E. Garc{\'i}a-R{\'i}o, {\em Maximum
      principles and gradient Ricci solitons,} J. Diff. Equations, {\bf 251}, (2011), 73-81.
\bibitem[GPP]{GPP} G. Gibbons, D. Page and C.N. Pope, {\em Einstein
    metrics on $\R^3$, $\R^4$ and $S^3$ bundles}, Comm. Math.
    Phys., {\bf 127}, (1990), 529-553.
\bibitem[Ha1]{Ha1} R. S. Hamilton, {\em The Ricci flow on surfaces}, in
     Mathematics and General Relativity (Santa Cruz, CA, 1986), Contemp Math.,
     {\bf 71}, Amer. Math. Soc., (1988), 237-262.
\bibitem[Ha2]{Ha2} R. S. Hamilton, {\em Eternal solutions to the Ricci
     flow}, Jour. Diff. Geom., {\bf 38}, (1993), 1-11.
\bibitem[HHR]{HHR} M. Hill, M. Hopkins and D. Ravenel, {\em On the non-existence
    of elements of Hopf invariant one,} arXiv:0908.3724v2.
\bibitem[Iv]{Iv} T. Ivey, {\em New examples of complete Ricci solitons},
       Proc. AMS, {\bf 122}, (1994), 241-245.
\bibitem[Ko]{Ko} N. Koiso, {\em On Rotationally symmetric Hamilton's equation
     for K\"ahler-Einstein metrics}, Adv. Studies Pure Math., {\bf 18-I},
     Academic Press, (1990), 327-337.
\bibitem[KS]{KS} S. Kwaski and R. Schultz, {\em Multiplication stablization and
      transformation groups,} in Current Trends in Transformation Groups,
      K-Monogr. Math., Kluwer, (2002), 147-165.
\bibitem[MS]{MS} O. Munteanu and N. Sesum, {\em On gradient Ricci solitons},
      J. Geom. Anal. {\bf 23}, (2013), no. 2, 539-561.
\bibitem[Per]{Per} G. Perelman, {\em The entropy formula for the Ricci flow and
       its geometric applications}, arXiv:math.DG/0211159.
\bibitem[PRS]{PRS} S. Pigola, M. Rimoldi and A. Setti, {\em Remarks on non-compact gradient
     Ricci solitons}, Math. Z., {\bf 268}, (2011), 777-790.
\bibitem[PS]{PS} F. Podesta and A. Spiro, {\em K\"ahler-Ricci solitons on
       homogeneous toric bundles}, J. reine angew. Math, {\bf 642}, (2010),
        109-127.
\bibitem[S]{S} R. Schultz, {\em Smooth structures on $S^p x S^q$,} Ann. Math.,
        {\bf 90}, (1969), 187-198.
\bibitem[WZh]{WZh} Xu-Jia Wang and Xiaohua Zhu, {\em K\"ahler-Ricci solitons on
      toric manifolds with positive first Chern class,} Adv. Math.,
      {\bf 188}, (2004), 87-103.
\bibitem[WeiWu]{WeiWu} G. Wei and P. Wu, {\em On volume growth of gradient
      steady Ricci solitons}, arXiv:math.DG/1208.2040.
\bibitem[Wu]{Wu} P. Wu, {\em On potential function of gradient steady
       Ricci solitons}, J. Geom. Anal., arXiv:math.DG/1102.3018.
\end{thebibliography}
\end{document}